\documentclass[reqno,11pt]{amsart}

\usepackage{amsmath,amssymb,latexsym,soul,cite,mathrsfs,amsfonts}
\usepackage{color,enumitem,graphicx}
\usepackage[colorlinks=true,urlcolor=blue,
citecolor=red,linkcolor=blue,linktocpage,pdfpagelabels,
bookmarksnumbered,bookmarksopen]{hyperref}
\usepackage[english]{babel}
\usepackage{comment}
\usepackage[left=2.9cm,right=2.9cm,top=2.8cm,bottom=2.8cm]{geometry}
\usepackage[hyperpageref]{backref}

\usepackage[colorinlistoftodos]{todonotes}
\makeatletter
\providecommand\@dotsep{5}
\def\listtodoname{List of Todos}
\def\listoftodos{\@starttoc{tdo}\listtodoname}
\makeatother

\usepackage{verbatim} 

\numberwithin{equation}{section}

\newtheorem{theorem}{Theorem}[section]
\newtheorem{proposition}[theorem]{Proposition}
\newtheorem{lemma}[theorem]{Lemma}
\newtheorem{corollary}[theorem]{Corollary}
\newtheorem{remark}[theorem]{Remark}

\newcommand\restr[2]{{
  \left.\kern-\nulldelimiterspace 
  #1 
  \vphantom{\big|} 
  \right|_{#2} 
  }}

\title[Local minimizers for a class of functionals over the Nehari set]
{Local minimizers for a class of functionals over the Nehari set}

\author[H. Ramos Quoirin]{Humberto Ramos Quoirin}

\author[K. Silva]{ Kaye Silva}

\address{H. Ramos Quoirin  \newline \indent CIEM-FaMAF \newline \indent Universidad Nacional de C\'{o}rdoba, (5000)
	C\'{o}rdoba, Argentina}
\email{\tt humbertorq@gmail.com}

\address{K. Silva \newline\indent
	Instituto de Matem\'atica e Estat\'istica.   
	\newline\indent 
	Universidade Federal de Goi\'as,
	\newline\indent
Rua Samambaia, 74001-970, Goi\^ania, GO, Brazil}
\email{\href{mailto:kayesilva@ufg.br}{kayesilva@ufg.br}}

\thanks{Kaye Silva was partially supported by CNPq/Brazil under Grant 408604/2018-2.}


\begin{document}

\begin{abstract}
We analyze the topological structure of the Nehari set for a class of functionals depending on a real parameter $\lambda$, and having two degrees of homogeneity. A special attention is paid to the extremal parameter $\lambda^*$, which is the threshold value for the Nehari set to be given by a {\it natural constraint}. The main difficulty arises when $\lambda>\lambda^*$, as the energy functional may be unbounded from below over the Nehari set. In such situation we prove the existence of local minimizers of the functional constrained to this set. We unify and extend previous existence and multiplicity results for critical points of indefinite, $(p,q)$-Laplacian, and Kirchhoff type problems.
\end{abstract}

\bigskip

\maketitle
\begin{center}
\begin{minipage}{12cm}
\tableofcontents
\end{minipage}
\end{center}

\bigskip

\maketitle
\section{Introduction}
\bigskip
This article is devoted to the analysis of the Nehari set (usually known as Nehari manifold) associated to a  functional depending on a real parameter. We shall proceed with the investigation on the extremal parameter  carried out in \cite{AS,IS}, where a class of indefinite and superlinear type problems has been investigated. We aim at unifying and extending the results of \cite{AS,IS} by dealing with a general class of functionals having two degrees of homogeneity. More precisely, we consider the family 
\begin{equation}\label{go}
\Phi_\lambda=\frac{1}{p}\left(P_1-\lambda P_2\right)-\frac{1}{\gamma}F,
\end{equation}
where $P_1,P_2,F$ are $C^1$ functionals acting on a uniformly convex Banach space $X$, $\lambda$ is a real parameter, and  $p,\gamma > 1$ with $p\neq \gamma$. Furthermore the following basic conditions shall be assumed:\\

\begin{itemize}
	\item $P_1,P_2$ are $p$-homogeneous and $F$ is $\gamma$-homogeneous, i.e. $P_1(tu)=t^p P_1(u)$, $P_2(tu)=t^pP_2(u)$ and $F(tu)=t^\gamma F(u)$ for all $t>0$ and $u \in X$.
	\item $P_2(u)\ge 0$ for all $u\in X\setminus\{0\}$ and $F(0)=0$.
	\item There exists $C_1,C_2,C_3>0$ such that $P_1(u)\ge C_1\|u\|^p$, $P_2(u)\le C_2\|u\|^p$ and $F(u)\le C_3\|u\|^\gamma$ for all $u\in X$.
	\item There exists $u_1,u_2,u_3\in X\setminus\{0\}$ such that $F(u_1)>0>F(u_2)$ and $F(u_3)=0$.
	\item $P_1$ is weakly lower semi-continuous, $P_2$ is weakly continuous, and $F$ is weakly upper semi-continuous.\\
\end{itemize} 

This class of functionals appears as the energy functional in several elliptic problems, among which the main prototype is 
\begin{equation}\label{fo}
\Phi_\lambda(u)=\frac{1}{p}\int_\Omega \left( |\nabla u|^p-\lambda |u|^p\right)-\frac{1}{\gamma}\int_\Omega f(x)|u|^{\gamma},
\end{equation}
defined for $u \in X=W_0^{1,p}(\Omega)$. Here $\Omega$ is a bounded domain of $\mathbb{R}^N$, $N \ge 1$, $p>1$ and $\gamma \neq p$ with $1<\gamma<p^*$ (the critical Sobolev exponent), and $f \in L^{\infty}(\Omega)$. The Euler-Lagrange equation for this functional is
\begin{equation}
\label{eo} -\Delta_p u =\lambda |u|^{p-2}u +f(x)|u|^{\gamma-2}u, \quad u \in W_0^{1,p}(\Omega),
\end{equation}
which has been studied by several authors \cite{AT,B,BD,BZ,DP,I,IS,Ou}. Inspired by the approach used in \cite{IS}, where the functional \eqref{fo} is considered in the superhomogeneous (or superlinear, if $p=2$) case $\gamma>p$, we shall investigate the general class \eqref{go}. In particular, we shall complement \cite{IS} by including the subhomogeneous case $\gamma<p$.
We shall also consider some variations of \eqref{eo}, namely, the Neumann problem
\[
\left\{
\begin{array}
[c]{lll}%
-\Delta_p u =\lambda |u|^{p-2}u +f(x)|u|^{\gamma-2}u & \mathrm{in} & \Omega,\\
\partial_n u=0 & \mathrm{on} & \partial\Omega,
\end{array}
\right.  
\]
where $\partial_n u$ denotes the outer normal derivative of $u$,
as well as problems with nonlinear boundary conditions.

Furthermore, our results also apply to the $(p,q)$-Laplacian problem
 \begin{equation*}
-\Delta_p u -\Delta_q u = \alpha |u|^{p-2}u+\beta |u|^{q-2}u, \quad u \in W_0^{1,p}(\Omega),
\end{equation*}
where $1<q<p$ and $\alpha,\beta \in \mathbb{R}$, and to the Kirchhoff type equation
\begin{equation*}
-\left(a+b\int_{\Omega}|\nabla u|^2 dx\right)\Delta u= \lambda u+\mu |u|^2u, \quad u \in H_0^1(\Omega),
\end{equation*}
where $a,b>0$, and $\lambda,\mu \in \mathbb{R}$.

Our main purpose is to analyze the topological structure (with respect to $\lambda$) of the Nehari set associated to $\Phi_\lambda$, which is defined by
$$\mathcal{N}_\lambda:=\{u \in X \setminus \{0\}: \Phi_\lambda'(u)u=0\}.$$
It can also be written as
\begin{equation*}
\mathcal{N}_\lambda=\{u\in X\setminus\{0\}:\varphi_{\lambda,u}'(1)=0\}=\{tu\in X\setminus\{0\}:\varphi_{\lambda,u}'(t)=0\},
\end{equation*}
where
$\varphi_{\lambda,u}:[0,\infty) \to \mathbb{R}$ is the {\it fibering map} given by $\varphi_{\lambda,u}(t)=\Phi_\lambda(tu)$ for $t \ge 0$ and $u \in X$.
A basic issue related to the Nehari set is to know whether it is given by a natural constraint, i.e. if any critical point of the restriction of $\Phi_\lambda$ to $\mathcal{N}_\lambda$ is a critical point of $\Phi_\lambda$.  
To discuss this issue, let us recall the splitting
\begin{equation*}
\mathcal{N}_\lambda=\mathcal{N}_\lambda^+\cup \mathcal{N}_\lambda^0\cup \mathcal{N}_\lambda^-,
\end{equation*}
where 
\begin{equation*}
\mathcal{N}_\lambda^+=\{u\in \mathcal{N}_\lambda : \varphi_{\lambda,u}''(1)>0\}, \quad \mathcal{N}_\lambda^-=\{u\in \mathcal{N}_\lambda : \varphi_{\lambda,u}''(1)<0\},
\end{equation*}
and
\begin{equation*}
\mathcal{N}_\lambda^0=\{u\in \mathcal{N}_\lambda : \varphi_{\lambda,u}''(1)=0\}.
\end{equation*}
are mutually disjoint sets. By the implicit function theorem, it is promptly seen that 
(whenever non-empty) $\mathcal{N}_\lambda^+$ and $\mathcal{N}_\lambda^-$ are $C^1$ manifolds in $X$, and critical points of $\Phi_\lambda $ restricted to $\mathcal{N}_\lambda^+\cup \mathcal{N}_\lambda^-$ are critical points of $\Phi_\lambda$.

In view of these facts we may refer to $\mathcal{N}_\lambda^+$ and $\mathcal{N}_\lambda^-$ as Nehari manifolds, and we see that $\mathcal{N}_\lambda$ is given by a natural constraint if and only if $\mathcal{N}_\lambda^0=\emptyset$. Let us note that most of the applications of the Nehari manifold method in the litterature (in particular the abstract results in  \cite{Am,SW}) occur when $\mathcal{N}_\lambda^0=\emptyset$, which prevents the difficulty previously described. This is the case in  \cite{BDH,B,BZ}, which deal with the functional \eqref{fo}.  The situation where $\mathcal{N}_\lambda^0\neq\emptyset$ also brings other difficulties, which are related with the behavior of $\Phi_\lambda$ on $\mathcal{N}_\lambda^+$ and  $\mathcal{N}_\lambda^-$. As a matter of fact,  in some situations we shall see that $\Phi_\lambda$ is unbounded from below on $\mathcal{N}_\lambda^+$ and its infimum over   $\mathcal{N}_\lambda^-$ is not achieved (see Lemma \ref{unboundedenergy} and Remark \ref{unbogamma<p}), which obviously makes impossible to use a standard minimization technique in these sets. Instead, we shall see that a local minimization procedure can be carried out in $\mathcal{N}_\lambda^+$.

Let us decribe in the sequel our main results. First we observe that $\mathcal{N}_\lambda^0$ becomes nonempty as soon as $\lambda$ crosses the threshold value
\begin{equation*}
\lambda^*:=\inf\left\{\frac{P_1(u)}{P_2(u)}: u \in X \setminus \{0\}, F(u)=0 \right\}.
\end{equation*}

Since we intend to minimize $\Phi_\lambda$ over $\mathcal{N}_\lambda^+$ and $\mathcal{N}_\lambda^-$, let us fix the following notation:
\begin{equation*}
c_\lambda^{\pm}:=\inf_{\mathcal{N}_\lambda^{\pm}} \Phi_\lambda.
\end{equation*}
Under the condition\\
\begin{itemize}
	\item[(H1)] $\lambda^*= \inf\left\{\frac{P_1(u)}{P_2(u)}: F(u)\ge 0\right\}$
\end{itemize}
it turns out that $c_\lambda^+$ or $c_\lambda^-$ provide a critical point of $\Phi_\lambda$
for $\lambda<\lambda^*$. Furthermore, in case $\lambda^*$ is larger than 
\begin{equation*}
\mu_*:=\inf\left\{\frac{P_1(u)}{P_2(u)}: u \in X, F(u)<0\right\}.
\end{equation*}
both $c_\lambda^+$ and $c_\lambda^-$ are achieved for $\mu_*<\lambda<\lambda^*$:

\begin{theorem}[Minimization up to $\lambda^*$]\label{existenceouyang} Assume that $\lambda<\lambda^*$. Then $\mathcal{N}_\lambda^0=\emptyset$. Moreover, under (H1) the following holds:
	\begin{enumerate}
		\item If $\gamma>p$   
		 then $c_\lambda^-$ is achieved, i.e. there exists $u_\lambda\in \mathcal{N}_\lambda^-$ such that $\Phi_\lambda(u_\lambda)=c_\lambda^->0$. If, in addition, $\mu_*<\lambda$ then $c_\lambda^+$ is achieved, i.e. there exists $w_\lambda\in \mathcal{N}_\lambda^+$ such that $\Phi_\lambda(w_\lambda)=c_\lambda^+<0$.
		\item If $\gamma<p$ 
		then $c_\lambda^+$ is achieved, i.e. there exists $u_\lambda\in \mathcal{N}_\lambda^+$ such that $\Phi_\lambda(u_\lambda)=c_\lambda^+<0$. If, in addition, $\mu_*<\lambda$ then $c_\lambda^-$ is achieved, i.e. there exists $w_\lambda\in \mathcal{N}_\lambda^-$ such that $\Phi_\lambda(w_\lambda)=c_\lambda^->0$.
	\end{enumerate}
\end{theorem}

For $\lambda \geq \lambda^*$ we shall be concerned only with minimization over $\mathcal{N}_\lambda^+$. Indeed, it turns out that $c_\lambda^-=0$ for $\lambda >\lambda^*$, see Lemma \ref{unboundedenergy} below. The following conditions play an important role in this case:\\
\begin{itemize}
	\item[(C1)] If  $\lambda^*$ is achieved by $u$ then $F'(u)\neq 0$.
	\item[(C2)] If $\lambda^*$ is achieved by $u$ then $H_{\lambda^*}'(u)\neq 0$.
	\item [(S)] If $w_n \rightharpoonup w$ in $X$ and $P_1(w_n) \to P_1(w)$, then $w_n\to w$ in $X$.\\
\end{itemize}

We point out that (S) is a structural condition needed in our minimization arguments, which is satisfied for instance if $P_1(u)=\|u\|^p$, in view of the uniform convexity of $X$.
On the other hand, (C1) and (C2) guarantee,  in combination with (H1), 
 that $\lambda^*$ is achieved by some $u \in \mathcal{N}_{\lambda}^0$, which up to some multiplicative constant yields a critical point of $\Phi_{\lambda^*}$. A second critical point may still be found in $\mathcal{N}_{\lambda^*}^+$:
 
\begin{theorem}[Minimization at $\lambda^*$]\label{existencelambda^*} Suppose (H1), (S),    (C1), (C2), and $\mu_*< \lambda^*$. Then:
\begin{enumerate}
\item   $\lambda^*$ is achieved and its minimizers satisfy $F(u)=0$. Moreover there exists $t>0$ satisfying $tu\in \mathcal{N}_{\lambda^*}^0$ and $\Phi_{\lambda^*}'(tu)=\Phi_{\lambda^*}(tu)=0$.
\item $c_{\lambda^*}^+$ is a critical value of $\Phi_{\lambda^*}'$, i.e. there exists $u\in \mathcal{N}_{\lambda^*}^+$ such that $\Phi_{\lambda^*}(u)=c_{\lambda^*}^+<0$ and $\Phi_{\lambda^*}'(u)=0$.
\end{enumerate}	
	 
\end{theorem}

Finally, for $\lambda$ larger than $\lambda^*$, the functional $\Phi_\lambda$ is no longer bounded from below on $\mathcal{N}_\lambda^+$, at least for $\gamma>p$ (see Lemma \ref{unboundedenergy} and Remark \ref{unbogamma<p}). Yet it has a local minimizer therein for $\lambda$ close to $\lambda^*$, which generates then a mountain-pass critical point:

\begin{theorem}[Minimization beyond $\lambda^*$] \label{t3} Suppose (H1), (S), (C1), (C2), and $\mu_*< \lambda^*$. Then there exists $\varepsilon>0$ having the following properties:  \begin{enumerate}
		\item $c_\lambda^+=-\infty$ and $\Phi_\lambda$ has a local minimizer  $u_\lambda \in \mathcal{N}_\lambda^+$ for $\lambda\in(\lambda^*,\lambda^*+\varepsilon)$. 
 \item Assume, in addition, the Palais-Smale condition:
	\begin{itemize}
		\item[$(PS)$]  If $(u_n) \subset X$ is a sequence such that $(\Phi_\lambda(u_n))$ 
		is bounded and $\Phi_\lambda'(u_n)\to 0$, then $(u_n)$ has a convergent subsequence.
	\end{itemize}
Then $\Phi_\lambda$ has a second critical point (of mountain-pass type) for $\lambda\in(\lambda^*,\lambda^*+\varepsilon)$.\\
	\end{enumerate}
\end{theorem}

\begin{remark}
Let us make some comments on the assumptions of Theorems \ref{existencelambda^*} and \ref{t3} in the context of our applications. In the case of the problem \eqref{eo}, many of these conditions  are satisfied if $\int_\Omega f(x)|\phi_1|^\gamma<0$, where $\phi_1$ is a positive eigenfunction associated to $\lambda_1(p)$, the first eigenvalue of the Dirichlet $p$-Laplacian. Indeed, in this case we clearly have $\mu_*=\lambda_1(p)<\lambda^*$. Moreover, the fact that $\lambda_1(p)$ is the only eigenvalue associated to a positive eigenfunction of the Dirichlet $p$-Laplacian implies that (H1) and (C2) satisfied. We refer to Section 3 for more details on this issue as well as the verification of these conditions for our further applications.
\end{remark}
The work is organized as follows: in section 2 we prove Theorems \ref{existenceouyang}, \ref{existencelambda^*} and \ref{t3}. In section 3 we apply these theorems to unify and extend previous existence and multiplicity results for critical points of indefinite, $(p,q)$-Laplacian, and Kirchhoff type problems.

\section{Proofs} 
To simplify the notation we set
$$H_\lambda:=P_1-\lambda P_2, \quad \mbox{i.e.} \quad
\Phi_\lambda=\frac{1}{p}H_\lambda -\frac{1}{\gamma}F.$$

\subsection{Basic properties}
The following result shall be used repeatedly:

\begin{lemma}\label{basic}
If $(u_n)\subset S$,  $\lambda_n \to \lambda \ge 0$, and $\limsup H_{\lambda_n}(u_n) \leq 0$ then, up to a subsequence, $u_n \rightharpoonup u$ in $X$ and $P_2(u)>0$. In particular $u \neq 0$.
\end{lemma}

\begin{proof}
Since $(u_n)$ is bounded and $X$ is reflexive, up to a subsequence, $u_n \rightharpoonup u$. If $P_2(u)\le 0$ then $H_\lambda(u)\leq \liminf H_{\lambda_n}(u_n) \leq  \limsup H_{\lambda_n}(u_n) \leq 0 \le P_1(u) \le H_\lambda(u)$, i.e. $H_{\lambda_n}(u_n) \to H_\lambda(u)=0$. Since $P_1=H_\lambda+\lambda P_2$ and $P_2(u_n)\to P_2(u)$, it follows that $P_1(u_n) \to \lambda P_2(u)\le 0$. But $P_1(u_n)\geq C_1>0$, and we reach a contradiction.
\end{proof}

Since $\Phi_\lambda$ is composed by two homogeneous terms, one may easily formulate a necessary and sufficient condition to have $\mathcal{N}_\lambda \setminus \mathcal{N}_\lambda^0\neq \emptyset$:
\begin{lemma}\label{neharisign} If $u\in \mathcal{N}_\lambda \setminus \mathcal{N}_\lambda^0$ then $H_\lambda(u)F(u)>0$. Conversely, if $H_\lambda(u)F(u)>0$  then there exists a unique $t=t_\lambda(u)>0$ such that $tu\in \mathcal{N}_\lambda  \setminus \mathcal{N}_\lambda^0$, which is given by $t_\lambda(u)=\left(H_\lambda(u)/F(u)\right)^{\frac{1}{\gamma-p}}$.
\end{lemma}
\begin{proof} If $u\in \mathcal{N}_\lambda$, then $J_\lambda(u)=0$ and thus $H_\lambda(u)=F(u)$. Conversely, if $H_\lambda(u)$ and $F(u)$ are nonzero and have the same sign, then the fibering map $\varphi_{\lambda,u}$ has a unique positive critical point $t$, so that $tu\in \mathcal{N}_\lambda$. The equation $\varphi_{\lambda,u}'(t)=0$ yields the desired expression of $t_\lambda(u)$.
\end{proof}
Let us set
\begin{equation*}
D_\lambda^+:=\{u\in X\setminus\{0\}: H_\lambda(u),F(u)>0\},
\end{equation*}
\begin{equation*}
D_\lambda^-:=\{u\in X\setminus\{0\}: H_\lambda(u),F(u)<0\},
\end{equation*}
and
\begin{equation*}
D_\lambda^0:=\{u\in X\setminus\{0\}: H_\lambda(u)=F(u)=0\}.
\end{equation*}
Since $H_\lambda$ and $F$ are homogeneous, we see that $D_\lambda^+$, $D_\lambda^-$ and $D_\lambda^0$ are cones, i.e. $u \in D_\lambda^+$ if and only if $tu \in D_\lambda^+$ for any $t>0$. The following properties are straightforward, so we omit their proofs:

\begin{proposition}\label{fiberingmaps} There holds $\mathcal{N}_\lambda^0= D_\lambda^0$. Furthermore:
	\begin{enumerate}
		\item If $\gamma>p$, then:
		\begin{enumerate}
			\item For each $u\in D_\lambda^+$, the  point $t_\lambda(u)$ is a non-degenerate global maximum point of $\varphi_{\lambda,u}$. Moreover $\mathcal{N}_\lambda^-=\{t_\lambda(u)u: u\in D_\lambda^+\}$. In particular $\mathcal{N}_\lambda^- \subset D_\lambda^+$.
			\item For each $u\in D_\lambda^-$, the  point $t_\lambda(u)$ is a non-degenerate global minimum point of $\varphi_{\lambda,u}$. Moreover $\mathcal{N}_\lambda^+=\{t_\lambda(u)u: u\in D_\lambda^-\}$. In particular $\mathcal{N}_\lambda^+ \subset D_\lambda^-$.\\
		\end{enumerate}
		\item If $\gamma<p$, then:
	\begin{enumerate}
		\item For each $u\in D_\lambda^+$, the  point $t_\lambda(u)$ is a non-degenerate global minimum point of $\varphi_{\lambda,u}$. Moreover $\mathcal{N}_\lambda^+=\{t_\lambda(u)u: u\in D_\lambda^+\}$. In particular $\mathcal{N}_\lambda^+ \subset D_\lambda^+$.
	\item 	For each $u\in D_\lambda^-$, the  point $t_\lambda(u)$ is a non-degenerate global maximum point of $\varphi_{\lambda,u}$. Moreover $\mathcal{N}_\lambda^-=\{t_\lambda(u)u: u\in D_\lambda^-\}$. In particular $\mathcal{N}_\lambda^- \subset D_\lambda^-$.\\
	\end{enumerate}
	\end{enumerate}
\end{proposition}
Recall that
\begin{equation*}
\lambda^*:=\inf\left\{\frac{P_1(u)}{P_2(u)}: u \in X \setminus \{0\}, F(u)=0 \right\},
\end{equation*}
and
\begin{equation*}
\mu_*:=\inf\left\{\frac{P_1(u)}{P_2(u)}: u \in X, F(u)<0\right\}.
\end{equation*}
We also introduce
\begin{equation*}
\mu^*:=\sup\left\{\frac{P_1(u)}{P_2(u)}: u \in X, F(u)>0\right\}.
\end{equation*}

Let us prove now that $\lambda^*$ is the threshold value for $\mathcal{N}_\lambda$ to be a manifold and that $\mu_*, \mu^*$ determine whether $\mathcal{N}_\lambda^+$ and $\mathcal{N}_\lambda^-$ are empty or not:
\begin{proposition}\label{neharinotempty} We have $\mathcal{N}_\lambda^0=\emptyset$ for any $\lambda<\lambda^*$. In addition:
	\begin{enumerate}
		\item If $\gamma>p$, then the following assertions hold:
		\begin{enumerate}
			\item $\mathcal{N}_\lambda^-\neq \emptyset $ if, and only if, $\lambda<\mu^*$.
			\item $\mathcal{N}_\lambda^+\neq \emptyset$ if, and only if, $\lambda>\mu_*$.
		\end{enumerate}
		\item If $\gamma<p$, then the following assertions hold:
		\begin{enumerate}
			\item $\mathcal{N}_\lambda^-\neq \emptyset $ if, and only if, $\lambda>\mu_*$.
		\item $\mathcal{N}_\lambda^+\neq \emptyset$ if, and only if, $\lambda<\mu^*$.
		\end{enumerate}
	\end{enumerate}
\end{proposition}
\begin{proof} 
	If $\lambda<\lambda^*$, then for any $u\in X\setminus\{0\}$ satisfying $F(u)=0$ we have $P_1(u)/P_2(u)>\lambda$ or, equivalently, $H_\lambda(u)>0$, so that $\mathcal{N}_\lambda^0=D_\lambda^0 = \emptyset$. 
	Let us assume that $\gamma>p$ and prove (1).\\
\begin{enumerate}
\item [(a)] Indeed, if $\lambda<\mu^*$ then there exists $u\in X\setminus\{0\}$ such that $F(u)>0$ and $P_1(u)/P_2(u)>\lambda$ or, equivalently, $H_\lambda(u)>0$ and thus, by Proposition \ref{fiberingmaps} it follows that $\mathcal{N}_\lambda^-\neq \emptyset$. Now, if $\lambda\ge \mu^*$, then it is clear that for every $u$ satisfying $F(u)>0$ we must have $H_\lambda(u)\le 0$, i.e. $D_\lambda^+ = \emptyset$, so that by Proposition \ref{fiberingmaps}, we conclude that $\mathcal{N}_\lambda^-= \emptyset$. \\
	
\item [(b)]	If $\lambda>\mu_*$ then there exists $u\in X\setminus\{0\}$ such that $F(u)<0$ and $P_1(u)/P_2(u)<\lambda$ or, equivalently, $H_\lambda(u)<0$ and thus, by Lemma \ref{neharisign} it follows that $\mathcal{N}_\lambda^+\neq \emptyset$. Now, if $\lambda<\mu_*$, then it is clear that for all $u$ satisfying $F(u)<0$ we must have $H_\lambda(u)\geq 0$, i.e. $D_\lambda^- = \emptyset$, so that by Proposition \ref{fiberingmaps}, we have $\mathcal{N}_\lambda^+= \emptyset$. \\ 
\end{enumerate}
The proof of (2) is completely similar, so we omit it.

\end{proof}

\subsection{Minimization up to $\lambda^*$}
In the sequel we shall use the assumption (H1), which we recall below:\\
\begin{itemize}
	\item[(H1)] $\lambda^*= \inf\left\{\frac{P_1(u)}{P_2(u)}: u \in X, F(u)\ge 0\right\}$.\\
\end{itemize}

 It is straightforward that this condition implies that $\lambda^*<\mu^*$ and
 the following properties:\\
\begin{equation}\label{eh1}
  H_{\lambda}(u)>0 \mbox{ for any $u \in X \setminus \{0\}$ such that } F(u) \geq 0, \mbox{ and any } \lambda<\lambda^*,
\end{equation}
and
\begin{equation}\label{eh11}
H_{\lambda^*}(u)\ge 0 \mbox{ for any $u \in X$ such that } F(u) \geq 0.\\
\end{equation}

The next result shows in particular that \eqref{eh1} actually holds in a stronger form under (H1), and that $N_\lambda^-$, $N_\lambda^+$ are away from zero and infinity, respectively, if $\lambda<\lambda^*$:
\begin{proposition}\label{coerciindefin} Suppose (H1) and $\lambda<\lambda^*$. Then: 
	\begin{enumerate}
		\item There exist $C,D>0$ such that	
		\begin{enumerate}
\item $H_\lambda(u)\ge C\|u\|^p$, for every $u\in X$ such that $F(u)\ge 0$.
\item $F(u)\le -D\|u\|^\gamma$ for every $u\in X$ such that $H_\lambda(u)\le 0$.
		\end{enumerate}
		\item  $\mathcal{N}_\lambda^-$ is away from zero, i.e. there exists $c>0$ such that $c\le \|u\|$ for all $u\in \mathcal{N}_\lambda^-$. Moreover $\Phi_\lambda$ is coercive over $\mathcal{N}_\lambda^-$.
		\item $\mathcal{N}_\lambda^+$ is bounded, i.e. there exists $C>0$ such that $\|u\|\le C$ for all $u\in \mathcal{N}_\lambda^+$.
	\end{enumerate}
\end{proposition}
\begin{proof} \strut
\begin{enumerate}
\item  We prove only (a), since (b) is similar. Suppose, on the contrary, that there exists a sequence $(u_n)\subset S$ such that $H_\lambda(u_n)<1/n$ amd $F(u_n)\ge 0$ for all $n$. By Lemma \ref{basic} we can assume that $u_n\rightharpoonup u \neq 0$ in $X$. Therefore
	\begin{equation*}
	H_\lambda(u)\le \liminf H_\lambda(u_n)\le0 \le \limsup F(u_n)\le F(u),
	\end{equation*}
	which contradicts \eqref{eh1}. \\

\item	 Indeed, note that
	\begin{equation}\label{nehariin}
	H_\lambda(u)=F(u), \quad \forall u\in \mathcal{N}_\lambda,
	\end{equation}
	and thus
	\begin{equation}\label{neharenergy}
	\Phi_\lambda(u)=\left(\frac{\gamma-p}{p\gamma}\right)H_\lambda(u)=\left(\frac{\gamma-p}{p\gamma}\right)F(u),\quad  \forall u\in \mathcal{N}_\lambda.\\
	\end{equation}
	\medskip
	
	{\bf Case 1: $\gamma>p$.}
Recall that $F(u)>0$ and $H_\lambda(u)>0$ for all $u\in \mathcal{N}_\lambda^-$. By (1) we conclude from \eqref{nehariin} that  
	\begin{equation*}
	C\|u\|^p\le 	H_\lambda(u)=F(u)\le C_3\|u\|^\gamma, \quad \forall u\in \mathcal{N}_\lambda^-.
	\end{equation*}
	and since $\gamma>p$, it follows that there exists $c>0$ such that $c\le\|u\|$ for all $u\in \mathcal{N}_\lambda^-$. Clearly by \eqref{neharenergy} and (1) we also have that $\Phi_\lambda$ is coercive.\\
	
	{\bf Case 2: $\gamma<p$.} Recall that $F(u)<0$ and $H_\lambda(u)<0$ for all $u\in \mathcal{N}_\lambda^-$. By (1) and \eqref{nehariin} we infer that  
	\begin{equation*}
	-d\|u\|^p\le H_\lambda(u)=F(u)\le -C\|u\|^\gamma, \quad \forall u\in \mathcal{N}_\lambda^-,
	\end{equation*}
	where $d>0$. Since $\gamma<p$, we obtain the desired conclusion.\\ 
\item  

	{\bf Case 1: $\gamma>p$.}
	Recall that $F(u)<0$ and $H_\lambda(u)<0$ for all $u\in \mathcal{N}_\lambda^+$.
	By (1) we conclude from \eqref{nehariin} that  
	\begin{equation*}
	-c\|u\|^p\le H_\lambda(u)=F(u)\le -C\|u\|^\gamma,  \quad \forall u\in \mathcal{N}_\lambda^+,
	\end{equation*}
	where $c>0$ and since $\gamma>p$, we obtain the desired inequality.\\

	{\bf Case 2: $\gamma<p$.}
Recall that $F(u)>0$ and $H_\lambda(u)>0$ for all $u\in \mathcal{N}_\lambda^+$.
	By (1) we conclude from \eqref{nehariin} that  
	\begin{equation*}
	C\|u\|^p\le 	H_\lambda(u)=F(u)\le C_3\|u\|^\gamma, \quad \forall u\in \mathcal{N}_\lambda^-.
	\end{equation*}
	and since $\gamma<p$, it follows that there exists $C>0$ such that $\|u\|\le C$ for all $u\in \mathcal{N}_\lambda^+$.
	\end{enumerate}	
\end{proof}

\noindent\textbf{Proof of Theorem \ref{existenceouyang}.}\\

Let $\lambda<\lambda^*$. By Proposition \ref{neharinotempty} we have $\mathcal{N}_\lambda^0=\emptyset$ so that $\mathcal{N}_\lambda$ is a $C^1$ manifold. Note also that $\lambda^* \leq \mu^*$, so according to Proposition \ref{neharinotempty} we have $\mathcal{N}_\lambda^-\neq \emptyset $ if $\gamma>p$ and $\mathcal{N}_\lambda^+\neq \emptyset $ if $\gamma<p$. If, in addition, $\lambda>\mu^*$ then $\mathcal{N}_\lambda^+\neq \emptyset $ if $\gamma>p$ and $\mathcal{N}_\lambda^-\neq \emptyset $ if $\gamma<p$. It remains to show that $c_\lambda^+$ and $c_\lambda^-$ are achieved whenever $\mathcal{N}_\lambda^+$ and $\mathcal{N}_\lambda^-$ are nonempty, respectively:\\

	\begin{enumerate}
\item First we deal with $c_{\lambda}^-$. Proposition \ref{fiberingmaps} yields that $c_{\lambda}^-\ge 0$. Let $(u_n)\subset \mathcal{N}_\lambda^-$ be a minimizing sequence for $c_{\lambda}^-$. By Proposition \ref{coerciindefin}(2) we can assume that $u_n\rightharpoonup u$ in $X$. We claim that $u\neq 0$.  Indeed, note by Proposition \ref{coerciindefin}(1) that
	\begin{equation}\label{eq1}
	C\|u\|^p\le\liminf C\|u_n\|^p\le \liminf H_\lambda(u_n)=\liminf F(u_n)\le F(u).
	\end{equation}
	Thus if $u=0$, then $u_n \to 0$, which contradicts Proposition \ref{coerciindefin}(2). Hence $u\neq 0$. Now we consider two cases:\\

	{\bf Case 1: $\gamma>p$.}
	We claim that $u\in D_\lambda^+$. In fact, note that $F(u)\ge \limsup F(u_n)\ge 0$, which implies that $H_\lambda(u)>0$ by \eqref{eh1}. From
	\begin{equation}\label{eq2}
	H_\lambda(u)\le \liminf H_\lambda(u_n)=\liminf F(u_n)\le F(u)
	\end{equation}
	 it follows that $F(u)>0$. Thus  Proposition \ref{fiberingmaps} provides us with  $t_\lambda(u)>0$ such that $t_\lambda(u)u\in \mathcal{N}_\lambda^-$. Therefore
	\begin{equation*}
	\Phi_\lambda(t_\lambda(u)u)\le \liminf \Phi_\lambda(t_\lambda(u)u_n)\le \liminf \Phi_\lambda(u_n)=c_{\lambda}^-,
	\end{equation*}
	and so $	\Phi_\lambda(t_\lambda(u)u)=c_{\lambda}^-$.\\

	{\bf Case 2: $\gamma<p$.}
	We claim that $u\in D_\lambda^-$. Indeed, note that $H_\lambda(u)\le \liminf H_\lambda(u_n)\le 0$, which implies that $F(u)<0$, by \eqref{eh1}. It follows from \eqref{eq2} that $H_\lambda(u)<0$. Thus by Proposition \ref{fiberingmaps}, there exists $t_\lambda(u)>0$ such that $t_\lambda(u)u\in \mathcal{N}_\lambda^-$. Therefore
	\begin{equation*}
	\Phi_\lambda(t_\lambda(u)u)\le \liminf \Phi_\lambda(t_\lambda(u)u_n)\le \liminf \Phi_\lambda(u_n)=c_{\lambda}^-,
	\end{equation*}
	and hence $	\Phi_\lambda(t_\lambda(u)u)=c_{\lambda}^-$.\\
	
\item Let us consider now $c_{\lambda}^+$. Propositions \ref{fiberingmaps} and \ref{coerciindefin}(3) yield that $-\infty<c_{\lambda}^+< 0$. Let $(u_n)\subset \mathcal{N}_\lambda^+$ be a minimizing sequence for $c_{\lambda}^+$. By Proposition \ref{coerciindefin}(3) we can assume that $u_n\rightharpoonup u$ in $X$ and it is clear that $u\neq 0$.  Once again we consider two cases:\\

	{\bf Case 1: $\gamma>p$.}
	We claim that $u\in D_\lambda^-$. Indeed, note that $H_\lambda(u)\le \liminf H_\lambda(u_n)\le 0$, and \eqref{eh1} implies that $F(u)<0$. From \eqref{eq2} it follows that $H_\lambda(u)<0$. Thus from Proposition \ref{fiberingmaps}, there exists $t_\lambda(u)>0$ such that $t_\lambda(u)u\in \mathcal{N}_\lambda^+$. Therefore
	\begin{equation*}
	\Phi_\lambda(t_\lambda(u)u)\le\Phi_\lambda(u) \le\liminf \Phi_\lambda(u_n)=c_{\lambda}^+,
	\end{equation*}
	and hence $	\Phi_\lambda(t_\lambda(u)u)=c_{\lambda}^+$.\\
	
	{\bf Case 2: $\gamma<p$.}
	We claim that $u\in D_\lambda^+$. Indeed, note that $F(u)\ge \limsup F(u_n)\ge 0$, which implies that $H_\lambda(u)>0$. By \eqref{eq2} it follows that $F(u)>0$. Thus from Proposition \ref{fiberingmaps}, there exists $t_\lambda(u)>0$ such that $t_\lambda(u)u\in \mathcal{N}_\lambda^+$. Therefore
\begin{equation*}
\Phi_\lambda(t_\lambda(u)u)\le\Phi_\lambda(u) \le\liminf \Phi_\lambda(u_n)=c_{\lambda}^+,
\end{equation*}
and therefore $	\Phi_\lambda(t_\lambda(u)u)=c_{\lambda}^+$.\qed\\
	\end{enumerate}

We conclude this subsection showing that for $\mu_*<\lambda<\lambda^*$ the minimization procedure over $\mathcal{N}_\lambda^+$ and $\mathcal{N}_\lambda^-$ is equivalent to minimizing $H_\lambda$ under the constraints $F(u)=\pm 1$. The latter method has been used to deal with the problem \eqref{eo} in \cite{Ou} for $\gamma>p=2$, and in \cite{KRQU} for $\gamma<p$.

\begin{lemma}\label{decreasingfibering}
	Let $\mu_*<\lambda<\lambda^*$ and $
	m^{\pm}_{\lambda}:=\inf\left\{H_\lambda(u): u \in X, F(u)=\pm 1 \right\}$. Then:
	\begin{enumerate}
		\item  $m_\lambda^{\pm}$ are achieved, and $m_\lambda^+>0>m_\lambda^-$.
		\item There holds \begin{equation}\label{cm}
		c_\lambda^+=
		\begin{cases}
		\frac{p-\gamma}{p\gamma}(-m_\lambda^-)^{\frac{\gamma}{\gamma-p}} & \mbox{ if } \gamma>p,\\
		\frac{\gamma-p}{p\gamma}(m_\lambda^+)^{\frac{\gamma}{\gamma-p}} & \mbox{ if }\gamma<p.
		\end{cases}
		\end{equation}
		and 
		\begin{equation}\label{cm2}
		c_\lambda^-=
		\begin{cases}
		\frac{\gamma-p}{p\gamma}(m_\lambda^+)^{\frac{\gamma}{\gamma-p}} & \mbox{ if } \gamma>p,\\
		\frac{p-\gamma}{p\gamma}(-m_\lambda^-)^{\frac{\gamma}{\gamma-p}} & \mbox{ if }\gamma<p.
		\end{cases}
		\end{equation}
		\item The maps $\lambda \mapsto m_\lambda^{\pm}$ are concave (therefore continuous) and decreasing in $(\mu_*,\lambda^*)$. In particular, the maps $\lambda \mapsto c_\lambda^{\pm}$ are decreasing and continuous in $(\mu_*,\lambda^*)$.
	\end{enumerate}
\end{lemma}

\begin{proof}\strut
	\begin{enumerate}
		\item Let us first show that $m^{\pm}_{\lambda}$ are finite. Indeed, if $(u_n) \subset X$ is such that
		$H_\lambda(u_n) \to -\infty$ and $F(u_n)=\pm 1$ then $(u_n)$ is unbounded, so we can assume that $\|u_n\| \to \infty$ and $v_n:=\frac{u_n}{\|u_n\|} \rightharpoonup v$. From $H_\lambda(v_n)\leq 0$ we have $v \neq 0$ and $H_\lambda(v)\leq 0 \leq F(v)$, which contradicts \eqref{eh1}. Thus $m^{\pm}_{\lambda}$ are both finite. The previous discussion also shows that any minimizing sequence $(u_n)$ for $m^{\pm}_{\lambda}$ is bounded, so we can assume that $u_n \rightharpoonup u$ in $X$. If $H_\lambda(u_n) \to m_\lambda^+$ and $F(u_n)=1$ then $H_\lambda(u)\leq m_\lambda^+$ and $F(u)\ge 1$. Hence $\tilde{u}:=\frac{u}{F(u)^{\frac{1}{\gamma}}}$ satisfies $F(\tilde{u})=1$ and $H_\lambda(\tilde{u})\leq H_\lambda(u)\leq m_\lambda^+$, i.e. it achieves $m_\lambda^+$. A similar argument shows that $m_\lambda^-$ is also achieved. Moreover, from \eqref{eh1} it is clear that  $m_\lambda^+>0$. Note also that $m_\lambda^-<0$ if and only if there exists $u \in X$ such that $F(u)=-1$ and $H_\lambda(u)<0$, which is equivalent to have $\lambda>\frac{P_1(u)}{P_2(u)}$ for some $u \in X$ such that $F(u)<0$, and this holds since $\lambda>\mu_*$.\\
		
		\item We prove \eqref{cm} for $\gamma>p$ (the case $\gamma<p$ and \eqref{cm2} are similar). Let $v_\lambda$ achieve $m_\lambda^-$, i.e. $H_\lambda(v_\lambda)=m_\lambda^-<0$ and $F(v_\lambda)=-1$. Thus $v_\lambda \in D_\lambda^-$ and $t_\lambda(v_\lambda)=(-m_\lambda^-)^{\frac{1}{\gamma-p}}$, so
		$$c_\lambda^+\leq \Phi_\lambda(t_\lambda(v_\lambda)v_\lambda)=\frac{1}{p}t_\lambda(v_\lambda)^p m_\lambda^- +
		\frac{1}{\gamma}t_\lambda(v_\lambda)^{\gamma}=\frac{p-\gamma}{p\gamma}(-m_\lambda^-)^{\frac{\gamma}{\gamma-p}}.$$
		On the other hand, since  $H(u_\lambda)=F(u_\lambda)<0$ we have $$m_\lambda^-\leq H_\lambda \left(\frac{u_\lambda}{(-F(u_\lambda))^{\frac{1}{\gamma}}}\right)=-(-H(u_\lambda))^{\frac{\gamma-p}{\gamma}}.$$
		It follows that $(-m_\lambda^-)^{\frac{\gamma}{\gamma-p}}\geq -H(u_\lambda)$,
		so that $c_\lambda^+=-\frac{p-\gamma}{p\gamma}H(u_\lambda)\geq \frac{p-\gamma}{p\gamma} (-m_\lambda^-)^{\frac{\gamma}{\gamma-p}}$.\\
		
		\item The concavity of $\lambda \mapsto m_\lambda^{\pm}$ follows from the fact that $m_\lambda^{\pm}$ are pointwise infima of $H_\lambda(u)$, which is affine (therefore concave) with respect to $\lambda$. Let $\mu_*<\lambda<\lambda'<\lambda^*$ and $v_\lambda$ achieve $m_\lambda^{\pm}$. Then 
		$$m_{\lambda'}^{\pm} \leq H_{\lambda'}(v_\lambda)=H_{\lambda}(v_\lambda)+(\lambda-\lambda')P_2(v_\lambda)<H_{\lambda}(v_\lambda)=m_{\lambda}^{\pm}.$$
		The assertions on $\lambda \mapsto c_\lambda^+$ follow from \eqref{cm}.\\
	\end{enumerate}	
\end{proof}

\subsection{Minimization at $\lambda^*$}

We start proving that $\lambda^*$ is achieved and provides us with a critical point of $\Phi_{\lambda^*}$ with zero energy if the following conditions are satisfied:\\
\begin{itemize}
	\item[(C1):] If  $\lambda^*$ is achieved by $u$, then $F'(u)\neq 0$.
	\item[(C2):] If $\lambda^*$ is achieved by $u$, then $H_{\lambda^*}'(u)\neq 0$.\\
\end{itemize}

\noindent\textbf{Proof of Theorem \ref{existencelambda^*} (1).}\\

First we show that $\lambda^*$ is achieved under (H1).
Since $P_1/P_2$ is $0$-homogeneous, we can find a minimizing sequence for $\lambda^*$ in $S$. By the weak (semi)continuity properties of $P_1$, $P_2$ and $F$, it follows that $\lambda^*$ is achieved by some $u$.
If $F(u)>0$ then $u$ is a local minimizer of $P_1/P_2$ over $X$, and consequently a critical point of this quotient, so that $P_1'(u)-\frac{P_1(u)}{P_2(u)}P_2'(u)=0$, i.e. $H'_{\lambda^*}(u)=0$, which contradicts $(C2)$. Thus $F(u)=0$.  By (C1) we may apply Lagrange's multiplier rule to find some $\alpha \in \mathbb{R}$ such that
	\begin{equation*}
	\left(\frac{P_1(u)}{P_2(u)}\right)'=\alpha F'(u),
	\end{equation*}
	which implies that
	\begin{equation}\label{compatu}
	H_{\lambda^*}'(u)=P_2(u)\alpha F'(u).
	\end{equation}
 $(C2)$ yields that $\alpha\neq 0$. We claim that $\alpha>0$. Otherwise  there exists $v\in X$ such that 	$H_{\lambda^*}'(u)v<0< F'(u)v$.
	Thus
	$$
	\lim_{h\to 0^+}\frac{H_{\lambda^*}(u+hv)}{h}=\lim_{h\to 0^+}\frac{H_{\lambda^*}(u+hv)-H_{\lambda^*}(u)}{h}<0$$
	and $$\lim_{h\to 0^+}\frac{F(u+hv)}{h}=\lim_{h\to 0^+}\frac{F(u+hv)-F(u)}{h}>0,$$
	which implies that there exists $\delta>0$ such that $H_{\lambda^*}(u+\delta v)<0<F(u+\delta v)$ and hence, for $\lambda<\lambda^*$ and close to $\lambda^*$ we conclude that $H_{\lambda}(u+\delta v)<0$ and $F(u+\delta v)>0$, a contradiction with \eqref{eh1}. Therefore $\alpha>0$  and setting $t=(\alpha P_2(u))^{-1/\gamma}$ we conclude from \eqref{compatu} that $tu$ is a critical point of $\Phi_{\lambda^*}$. Moreover it is clear that $tu\in \mathcal{N}_{\lambda^*}^0$ and $\Phi_{\lambda^*}(tu)=0$.\qed\\

Let us set $$\Theta_\lambda=\{u/\|u\|:u \in\mathcal{N}_\lambda^+\}.$$ From Proposition \ref{neharinotempty} we have $\Theta_\lambda\neq \emptyset$ for $\lambda>\mu_*$ if $\gamma>p$, and for $\lambda<\mu^*$ if $\gamma<p$. 
\begin{lemma}\label{projectionnehari} \strut
	\begin{enumerate}
		\item If $\gamma>p$ then $\Theta_\lambda$ is increasing with respect to $\lambda$, i.e.  $\Theta_\lambda \subset\Theta_{\lambda'}$ for $\mu_*<\lambda<\lambda'$.
			\item If $\gamma<p$ then $\Theta_\lambda$ is decreasing with respect to $\lambda$, i.e. $\Theta_{\lambda'} \subset\Theta_{\lambda}$ for  $\lambda<\lambda'<\mu^*$.
	\end{enumerate}
\end{lemma}
\begin{proof} Let $\gamma>p$ and $\mu_*<\lambda<\lambda'$. If $v\in \Theta_\lambda$, then $H_\lambda(v)<0$ and $F(v)<0$. Thus $H_{\lambda'}(v)<0$ and consequently $v\in \Theta_{\lambda'}$. The proof of $(2)$ is similar (recall that $H_{\lambda'}(v)>0$ in this case).
	
\end{proof}

%
\begin{lemma}\label{limit} Let $\lambda_n\nearrow \lambda^*$ and $u_n\in \mathcal{N}_{\lambda_n}^+$ be such that $\Phi_{\lambda_n}(u_n)=c_{\lambda_n}^+$. If $\gamma>p$ and $u_n \rightharpoonup u\in\mathcal{D}_{\lambda^*}^-$ (respect. $\gamma<p$ and $u_n \rightharpoonup u\in\mathcal{D}_{\lambda^*}^+$) then $u \in \mathcal{N}_{\lambda^*}^+$, $\Phi_{\lambda^*}(u)=c_{\lambda^*}^+$ and $\Phi_{\lambda^*}'(u)=0$. 
\end{lemma}
\begin{proof} It suffices to show that $\Phi_{\lambda^*}(u)=c_{\lambda^*}^+$. Since $t_{\lambda^*}(u)$ is the global minimum point of $\varphi_{\lambda^*,u}$ we have $\Phi_{\lambda^*}(u)\geq \Phi_{\lambda^*}(t_{\lambda^*}(u)u) \ge c_{\lambda^*}^+$. Suppose, by contradiction, that $\Phi_{\lambda^*}(u)>c_{\lambda^*}^+$.
		\vskip.3cm
	{\bf Case 1: $\gamma>p$}
	\vskip.3cm
	Given $\alpha\in (c_{\lambda^*}^+,\Phi_{\lambda^*}(u))$, we choose $v\in \Theta_{\lambda^*}$ such that $c_{\lambda^*}^+\le \Phi_{\lambda^*}(t_{\lambda^*}(v)v)< \alpha$. By continuity there exists $\delta>0$ such that $\Phi_{\lambda}(t_{\lambda}(v)v)<\alpha$ for all $\lambda\in (\lambda^*-\delta,\lambda^*)$. Therefore, for $n$ sufficiently large we have 
	\begin{equation*}
	 \Phi_{\lambda_n}(u_n)=c_{\lambda_n}^+\le \Phi_{\lambda_n}(t_{\lambda_n}(v)v)<\alpha <\Phi_{\lambda^*}(u),
	\end{equation*}
	which contradicts the fact that $\Phi_{\lambda^*}(u)\leq \liminf  \Phi_{\lambda_n}(u_n)$.\\

		{\bf Case 2: $\gamma<p$}
	\vskip.3cm
	Given $\alpha\in (c_{\lambda^*}^+,\Phi_{\lambda^*}(u))$, choose $v\in \Theta_{\lambda^*}$ such that $c_{\lambda^*}^+ \le \Phi_{\lambda^*}(t_{\lambda^*}(v)v)< \alpha$. Take $\delta>0$ such that $t_\lambda(v)$ is well defined and $\Phi_{\lambda}(t_{\lambda}(v)v)<\alpha$ for $\lambda\in (\lambda^*-\delta,\lambda^*)$. One may argue as in the previous case to reach a contradiction.
	
\end{proof}

We are now in position to provide the\\

\noindent\textbf{Proof of Theorem \ref{existencelambda^*} (2).}\\

 Choose a sequence $\lambda_n\nearrow \lambda^*$ and $u_n\in \mathcal{N}_{\lambda_n}^+$ such that $\Phi_{\lambda_n}(u_n)=c_{\lambda_n}^+$ (which exists by Theorem \ref{existenceouyang}). We write $u_n=t_nw_n$, where $w_n\in \Theta_{\lambda_n}$ and $t_n:=t_{\lambda_n}(w_n)$ is given by
	\begin{equation}\label{eq3}
	t_n=\left(\frac{H_{\lambda_n}(w_n)}{F(w_n)}\right)^{\frac{1}{\gamma-p}}.
	\end{equation}
	Note also that by Lemma \ref{decreasingfibering}(3) there exists $c<0$ such that
	\begin{equation}\label{functional}
 c>c_{\lambda_n}^+=\Phi_{\lambda_n}(t_nw_n)= \pm\frac{\gamma-p}{p\gamma}t_n^p|H_{\lambda_n}(w_n)|,
	\end{equation}
		where the sign $-$ is for the case $\gamma>p$ and the sign $+$ corresponds to $\gamma<p$. Let us show that $u_n \rightharpoonup u\in\mathcal{N}_{\lambda^*}^+$, so that Lemma \ref{limit} yields the desired conclusion:
		\vskip.3cm
	{\bf Case 1: $\gamma>p$}
	\vskip.3cm
	Since $H_{\lambda_n}(w_n)<0$, by Lemma \ref{basic} we can assume that  $w_n\rightharpoonup w \neq 0$ in $X$. We claim that $(t_n)$ is bounded and away from zero. Indeed, suppose by contradiction that, up to a subsequence $t_n \to \infty$. Then, by \eqref{eq3}, we conclude that $F(w_n)\to 0$ and hence $F(w)\ge\limsup F(w_n)= 0$. Since $H_{\lambda^*}(w)\le\liminf H_{\lambda_n}(w_n)\le 0$, we deduce from \eqref{eh11} that $H_{\lambda^*}(w)=0$. It follows from Theorem \ref{existencelambda^*}(1)  that $H_{\lambda^*}(w)=F(w)=0$ and $w$ achieves $\lambda^*$. Moreover since 
	\begin{equation*}
	0= H_{\lambda^*}(w)\le \liminf H_{\lambda_n}(w_n)\le 0,
	\end{equation*}
	 we have $H_{\lambda_n}(w_n)\to 0$, so $P_1(w_n) \to P_1(w)$. Hence, by (S2) we obtain $w_n\to w$ in $X$,  and from
		\begin{equation*}
	t_n^{p-1}H'_{\lambda_n}(w_n)=t_n^{\gamma-1} F'(w_n),\quad \forall n,
	\end{equation*}
	we  infer that $F'(w)=0$, which contradicts $(C1)$. So $(t_n)$ is bounded. Now it is clear from \eqref{functional} that $t_n \not \to0$, so the claim is proved and we can suppose that $t_n \to t\in (0,\infty)$. 
	Thus $u_n=t_nw_n \rightharpoonup u:=tw$.
	Note again by \eqref{functional} that $H_{\lambda^*}(u)=t^pH_{\lambda^*}(w)\leq t^p\liminf H_{\lambda_n}(w_n)<0 $ and consequently $F(u)<0$, i.e. $u \in \mathcal{D}_{\lambda^*}^-$. The conclusion follows from Lemma \ref{limit}.
	\vskip.3cm
		{\bf Case 2: $\gamma<p$}
	\vskip.3cm
	 Let us show again that $(t_n)$ is bounded and away from zero. Indeed, suppose that up to a subsequence $t_n \to \infty$, so that by \eqref{eq3} we conclude that $H_{\lambda_n}(w_n)\to 0$ and thus by Lemma \ref{basic} we can assume that $w_n\rightharpoonup w \neq 0$ in $X$. Since $F(w)\ge\limsup F(w_n)= 0$ and $H_{\lambda^*}(w)\le\liminf_{n\to \infty}H_{\lambda_n}(w_n)\le 0$, we deduce from \eqref{eh11} and Theorem \ref{existencelambda^*}(1)  that $H_{\lambda^*}(w)=F(w)=0$ and therefore $w$ achieves $\lambda^*$. Moreover $w_n\to w$ in $X$ and since
\begin{equation*}
t_n^{p-1}H'_{\lambda_n}(w_n)=t_n^{\gamma-1} F'(w_n),  \quad \forall n,
\end{equation*}
we  conclude that $H_{\lambda^*}'(w)=0$, which contradicts $(C2)$. Thus $(t_n)$ is bounded, 
and the proof can then be concluded as in the previous case.

\subsection{Local Minimization Beyond $\lambda^*$} In this Section we look for critical points of $\Phi_\lambda$ for $\lambda>\lambda^*$. We start with the following observation, which shows that $c_\lambda^+$ (for $\gamma>p$) and $c_\lambda^-$ (for $\gamma<p$) are no longer achieved. 
\begin{lemma}\label{unboundedenergy} Suppose (H1), (C1), and (C2). 
	\begin{enumerate}
		\item If $\gamma>p$, then  $c_\lambda^+=-\infty$ for all $\lambda > \lambda^*$.
		\item If $\gamma<p$, then  $c_\lambda^-=0$ for all $\lambda> \lambda^*$.
	\end{enumerate}
\end{lemma}
\begin{proof}\strut 
	\begin{enumerate} 
	\item	By Theorem \ref{existencelambda^*}(1) there exists $u\in X$ such that $H_{\lambda^*}(u)=F(u)=0$ and
	\begin{equation}\label{eq7}
	\frac{1}{p}H_{\lambda^*}'(u)=\frac{1}{\gamma}F'(u)\neq 0.
	\end{equation}
	We choose  $v\in X$ such that $H_{\lambda^*}'(u)v<0$ and $F'(u)v<0$. As in the proof of Theorem \ref{existencelambda^*}(1), we have
	\begin{equation*}
	H_{\lambda^*}(u+sv)<0 \ \ \mbox{and}\ \ F(u+sv)<0,
	\end{equation*}
	for $s>0$ small enough.
	Now fix $\lambda>\lambda^*$ and note that $H_\lambda(u+sv)<0$, which implies, in particular, that $u+sv\in D_\lambda^-$. Moreover, by continuity  $H_\lambda(u+sv)\to H_\lambda(u)<H_{\lambda^*}(u)=0$ as $s\to 0^+$. Therefore
	  $t_{\lambda}(u+sv)(u+sv)\in \mathcal{N}_\lambda^+$  and
		\begin{equation*}
			\lim_{s\to 0^+}\Phi_\lambda(t_{\lambda}(u+sv)(u+sv))=\lim_{s\to 0^+}-\frac{\gamma-p}{\gamma p}\frac{|H_\lambda(u+sv)|^{\frac{\gamma}{\gamma-p}}}{|F(u+sv)|^{\frac{p}{\gamma-p}}}=-\infty.
		\end{equation*}
	\item The argument is similar to the previous one. Note that now if $u+sv\in D_\lambda^{\pm}$ then 
	 $t_{\lambda}(u+sv)(u+sv)\in \mathcal{N}_\lambda^{\pm}$ and $\gamma-p<0$.
	\end{enumerate}
\end{proof}

\begin{remark}\label{unbogamma<p} \strut \rm 
\begin{enumerate} 
\item Proving that $c_\lambda^+=-\infty$ for $\lambda > \lambda^*$ when $\gamma<p$ (as well as $c_\lambda^-=0$ for $\lambda> \lambda^*$ when $\gamma>p$) is more delicate and we are not able to do it in this general setting. However, this is indeed the case in our applications  (see Remark \ref{unboundedexamples}). Finally, let us observe that this result can be proved if we assume that the set $\{u \in X: F(u)>0\}$ is pathwise connected.   
\item An argument similar to the one in the previous proof shows that $c_{\lambda^*}^-=0$ for $\gamma<p$.
\end{enumerate}	
\end{remark}

To overcome the problem posed by Lemma \ref{unboundedenergy}, we show that $\Phi_\lambda$ has a local minimizer over $\mathcal{N}_\lambda^+$. Let us first extend Proposition \ref{coerciindefin}:
\begin{proposition}\label{coercivi>lambda^*} Suppose (H1), (C1), (C2), and $\mu_*< \lambda^*$.
	\begin{enumerate}
		\item Given $\mu>0$ there exist $C,\varepsilon>0$  such that
		$H_\lambda(u)\ge C\|u\|^p$ for any $u\in X$ satisfying $F(u)\ge \mu\|u\|^{\gamma}$, and any $\lambda\in[\lambda^*,\lambda^*+\varepsilon]$.
		\item Given $\mu\in(\mu_*,\lambda^*)$  there exists $D>0$ such that	$F(u)\le -D\|u\|^{\gamma}$ for any $u\in X$ such that $H_\mu(u)\le 0$.
	\end{enumerate}
\end{proposition} 
\begin{proof} \strut
\begin{enumerate}
\item Suppose, on the contrary, that there exists a sequence $\lambda_n \to \lambda^*$ with $\lambda_n > \lambda^*$ and $(u_n) \subset X$ such that $F(v_n) \ge \mu$ and $H_{\lambda_n}(v_n)\to 0$, where  $v_n=u_n/\|u_n\|$, for every $n$. Thus $v_n \rightharpoonup v \neq 0$,   by Lemma \ref{basic}. Therefore $H_{\lambda^*}(v)\le 0 <\mu \le F(v)$ and by definition of $\lambda^*$ we have $H_{\lambda^*}(v)= 0$, i.e. $\lambda^*$ is achieved by $v$. Theorem \ref{existencelambda^*} (1) implies that $F(v)= 0$, which yields a contradiction.\\

\item  Arguing by contradiction we find a sequence $(u_n)\subset S$ such that $H_\mu(u_n)\le 0$ and $F(u_n)\ge -1/n$ for all $n$. By Lemma \ref{basic} we can assume that $u_n\rightharpoonup u \neq 0$ in $X$. Therefore $H_\mu(u)\le 0 \le F(u)$, which contradicts \eqref{eh1}.
\end{enumerate}
\end{proof}
Next we fix $(\lambda,\mu)\in(\lambda^*,\infty)\times (\mu_*,\lambda^*)$ and set
\begin{equation*}
\mathcal{N}_{\lambda,\mu}^+:=
\begin{cases}
\{u\in \mathcal{N}_\lambda^+: H_\mu(u)<0\} & \mbox{ if } \gamma>p,\\
\{u\in \mathcal{N}_\lambda^+: F(u/\|u\|)>\mu\} & \mbox{ if }\gamma<p.
\end{cases}
\end{equation*}

In the next result $\varepsilon$ is given by Proposition \ref{coercivi>lambda^*}.

\begin{proposition}\label{boundecoerci>lambda^*} Suppose (H1), (C1), (C2), and $\mu_*< \lambda^*$. Then $ \mathcal{N}_{\lambda,\mu}^+$ is bounded for $\lambda\ge \lambda^*$ if $\gamma>p$ (respect. for $\lambda\in[\lambda^*,\lambda^*+\varepsilon]$ if $\gamma<p$).
\end{proposition}

\begin{proof} Let $\gamma>p$. By using Proposition \ref{coercivi>lambda^*} (2), we can argue as in the proof of Proposition \ref{coerciindefin}(3). Now, if $\gamma<p$ then we argue by contradiction. Assume that there exist $\lambda_n \to \lambda^*$ with $\lambda_n > \lambda^*$ and $\mathcal{N}_{\lambda_n}^+$ unbounded. So we can find a sequence $(u_n)$ such that $u_n \in \mathcal{N}_{\lambda_n,\mu}^+$ for every $n$, and $\|u_n\| \to \infty$. Setting $v_n=\frac{u_n}{\|u_n\|}$,  we may assume that $v_n \rightharpoonup v$ in $X$. Since $t_{\lambda_n}(v_n)=\|u_n\|$ we have that
	$$t_{\lambda_n}(v_n)=\left(\frac{F(v_n)}{H_{\lambda_n}(v_n)}\right)^{\frac{1}{p-\gamma}} \to \infty.$$
	It follows that  $H_{\lambda_n}(v_n) \to 0$. On the other hand, since $F(v_n)>\mu$ we reach a contradiction as in the proof of Proposition \ref{coercivi>lambda^*} (1).

\end{proof}
Let us now introduce
\begin{equation*}
c_{\lambda,\mu}^+:=\inf_{\mathcal{N}_{\lambda,\mu}^+}\Phi_\lambda .
\end{equation*}
\begin{theorem}\label{existenceouyanggeneral} Suppose (H1), (C1), (C2), and $\mu_*< \lambda^*$.
	\begin{enumerate}
		\item If $\gamma>p$ and $\lambda\ge \lambda^*>\mu>\mu_*$, then there exists $u_\lambda\in \overline{\mathcal{N}_{\lambda,\mu}^+}$ such that $\Phi_\lambda(u_\lambda)=c_{\lambda,\mu}^+<0$.
		\item If $\gamma<p$ and  $\lambda\in[\lambda^*,\lambda^*+\varepsilon]$, then there exists $u_\lambda\in \overline{\mathcal{N}_{\lambda,\mu}^+}$ such that $\Phi_\lambda(u_\lambda)=c_{\lambda,\mu}^+<0$.
	\end{enumerate}
\end{theorem}
\begin{proof}  Propositions \ref{fiberingmaps} and \ref{boundecoerci>lambda^*} imply that $-\infty<c_{\lambda,\mu}^+< 0$. Let $(u_n)\subset \mathcal{N}_{\lambda,\mu}^+$ be a minimizing sequence for $c_{\lambda,\mu}^+$. By Proposition \ref{boundecoerci>lambda^*} we can assume that $u_n\rightharpoonup u$ in $X$ and it is clear that $u\neq 0$. 
	
	\begin{enumerate}
\item 	 Note that $H_\mu(u)\le \liminf H_\mu(u_n)\le 0$, which implies that $F(u)<0$ by \eqref{eh1}. Since $H_\lambda(u)<H_{\mu}(u)\le0$, by Proposition \ref{fiberingmaps} there exists $t_\lambda(u)>0$ such that $t_\lambda(u)u\in \mathcal{N}_{\lambda}^+$, and then $t_\lambda(u)u\in \overline{\mathcal{N}_{\lambda,\mu}^+}$ since $H_\mu(t_\lambda(u)u)=t_\lambda(u)^pH_\mu(u)\le 0$.  Therefore
	\begin{equation*}
	\Phi_\lambda(t_\lambda(u)u)\le\Phi_\lambda(u) \le\liminf \Phi_\lambda(u_n)=c_{\lambda,\mu}^+,
	\end{equation*}
	i.e. $	\Phi_\lambda(t_\lambda(u)u)=c_{\lambda,\mu}^+$.\\

\item Note that $F(u/\|u\|)\ge \limsup F(u_n/\|u_n\|)\ge \mu$, which implies by Proposition \ref{coercivi>lambda^*} that $H_\lambda(u/\|u\|)>0$. Thus by Proposition \ref{fiberingmaps}, there exists $t_\lambda(u)>0$ such that $t_\lambda(u)u\in  \overline{\mathcal{N}_{\lambda,\mu}^+}$. Therefore
\begin{equation*}
\Phi_\lambda(t_\lambda(u)u)\le\Phi_\lambda(u) \le\liminf \Phi_\lambda(u_n)=c_{\lambda,\mu}^+,
\end{equation*}
so that $	\Phi_\lambda(t_\lambda(u)u)=c_{\lambda,\mu}^+$.
\end{enumerate}
\end{proof}

 We introduce now the set of minimizers associated to $c_\lambda^+$, i.e.
\begin{equation*}
\mathcal{S}_\lambda=\{u\in \mathcal{N}_\lambda^+:\Phi_\lambda(u)=c_\lambda^+\}.
\end{equation*}
\begin{lemma} \label{compact} Suppose (H1), $\mu_*< \lambda^*$, and (C1) if $\gamma>p$ (respect. (C2) if $\gamma<p$). Then $\mathcal{S}_{\lambda^*}$ is compact.
\end{lemma}
\begin{proof} Indeed, take $(u_n)\subset \mathcal{S}_{\lambda^*}$ and note that $\Phi_{\lambda^*}'(u_n)=0$. Writing $u_n=t_nw_n$ with $w_n\in \Theta_{\lambda^*}$, and arguing as in the proof of Theorem \ref{existencelambda^*}, we conclude that up to a subsequence $u_n \to u\in \mathcal{N}_{\lambda^*}^+$ and the proof is complete.  
\end{proof}
\begin{corollary}\label{submanifold} Suppose (H1), $\mu_*< \lambda^*$, and $\mathcal{S}_{\lambda^*}\neq\emptyset$.
	\begin{enumerate}
		\item Assume $\gamma>p$ and (C1). Then there exists $\mu\in(\mu_*,\lambda^*)$ such that $H_\mu(u)<0$ for all $u\in \mathcal{S}_{\lambda^*}$.
		\item Assume $\gamma<p$ and (C2). Then there exists $\mu>0$ such that $F(u)>\mu$ for all $u\in \mathcal{S}_{\lambda^*}$.
	\end{enumerate} 
\end{corollary}
\begin{proof} \strut
	\begin{enumerate}
		\item By Lemma \ref{compact} the set $\mathcal{S}_{\lambda^*}$ is compact and $H_{\lambda^*}(u)=\frac{p\gamma}{\gamma-p}c_{\lambda^*}<0$ for all $u\in \mathcal{S}_{\lambda^*}$,. Now choose $\mu\in(\mu_*,\lambda^*)$ such that $H_\mu(u)<0$ for all $u\in \mathcal{S}_{\lambda^*}$ and the proof is complete.\\
		
		\item Indeed, we have $F(u)=\frac{p\gamma}{\gamma-p}c_{\lambda^*}>0$ for all $u\in \mathcal{S}_{\lambda^*}$.
	\end{enumerate}		
\end{proof}

Next we deal with $\mu$  given by Corollary \ref{submanifold} and $u_\lambda$  given by Theorem \ref{existenceouyanggeneral}:
\begin{lemma}\label{limit1} Let $\lambda_n\searrow \lambda^*$  and $u_n\in \overline{\mathcal{N}_{\lambda_n,\mu}^+}$ be such that $\Phi_{\lambda_n}(u_n)=c_{\lambda_n,\mu}^+$.  If $\gamma>p$ and $u_n \rightharpoonup u\in\mathcal{D}_{\lambda^*}^-$ (respect. $\gamma<p$ and $u_n \rightharpoonup u\in\mathcal{D}_{\lambda^*}^+$) then $u_n \to u \in \mathcal{S}_{\lambda^*}$.
\end{lemma}
\begin{proof} We argue as in the proof of Lemma \ref{limit}. First note that $\Phi_{\lambda^*}(u)\ge c_{\lambda^*}^+$. Suppose, by contradiction, that $\Phi_{\lambda^*}(u)>c_{\lambda^*}^+$. Recall from Theorem \ref{existencelambda^*} that there exists $v\in \mathcal{S}_{\lambda^*}$, i.e. $v \in \mathcal{N}_{\lambda^*}^+$ with $c_{\lambda^*}^+= \Phi_{\lambda^*}(v)$.
	\vskip.3cm
	{\bf Case 1: $\gamma>p$}
	\vskip.3cm
	 By Corollary \ref{submanifold} we  have $H_\mu(v)<0$. Then $w:=v/\|v\| \in \Theta_{\lambda^*}$ and $H_\mu(w)<0$.  From Lemma \ref{decreasingfibering}(3) we know that $\Phi_{\lambda}(t_\lambda(w)w)<c_{\lambda^*}^+$ for all $\lambda>\lambda^*$. Since $t_{\lambda_n}(w)w \in \mathcal{N}_{\lambda_n,\mu}^+$ we find that
	\begin{equation*}
	\Phi_{\lambda_n}(u_n)=c_{\lambda_n,\mu}^+\le \Phi_{\lambda_n}(t_{\lambda_n}(w)w)<c_{\lambda^*}^+ <\Phi_{\lambda^*}(u),
	\end{equation*}
	which contradicts $\Phi_{\lambda^*}(u)\leq \liminf  \Phi_{\lambda_n}(u_n)$. Therefore $u \in \mathcal{S}_{\lambda^*}$, and repeating the argument above with $u$ instead of $v$, we find that $\Phi_{\lambda_n}(u_n) \to  \Phi_{\lambda^*}(u)$, so that $P_1(u_n) \to P_1(u)$. Condition (S2) implies that $u_n \to u$ in $X$.
	\\
	
	{\bf Case 2: $\gamma<p$}
	\vskip.3cm
	 We have now $\limsup F(t_{\lambda_n}(w)w)=\limsup t_{\lambda_n}(w)^\gamma F(w)\ge F(v)>\mu$, i.e. $t_{\lambda_n}(w)w \in \mathcal{N}_{\lambda_n,\mu}^+$. From Lemma \ref{decreasingfibering}(3) we know that there exists $\varepsilon>0$ such that $\Phi_{\lambda}(t_\lambda(w)w)<c_{\lambda^*}^+$ for all $\lambda\in(\lambda^*,\lambda^*+\varepsilon)$. One can argue as in the previous case to reach a contradiction.
	
\end{proof}

\noindent\textbf{Proof of Theorem \ref{t3} (1).}\\

Let $u_\lambda$ be given by Theorem \ref{existenceouyanggeneral}, i.e. $u_\lambda\in \overline{\mathcal{N}_{\lambda,\mu}^+}$ satisfies $\Phi_\lambda(u_\lambda)=c_{\lambda,\mu}^+<0$.
We claim that there exists $\varepsilon>0$ such that $u_\lambda\in \mathcal{N}_{\lambda,\mu}^+$ for all  $\lambda\in (\lambda^*,\lambda^*+\varepsilon)$, i.e.  $c_{\lambda,\mu}^+$ is achieved for these values of $\lambda$.

			\vskip.3cm
	{\bf Case 1: $\gamma>p$}
	\vskip.3cm
	
	Let us prove that $H_\mu(u_\lambda)<0$ if $\lambda\in (\lambda^*,\lambda^*+\varepsilon)$, for some $\varepsilon>0$. Indeed, suppose on the contrary, that there exists a sequence $\lambda_n \searrow \lambda^*$ such that $u_n:=u_{\lambda_n}$ satisfies $H_{\mu}(u_n)= 0$ and $\Phi_{\lambda_n}(u_n)=c_{\lambda_n,\mu}^+$.
	As in the proof of Theorem \ref{existencelambda^*}(2), we can show that writing $u_n=t_nw_n$, where $w_n\in \Theta_{\lambda_n}$ and $t_n:=t_{\lambda_n}(w_n)$, up to a subsequence  $u_n\rightharpoonup u  \in \mathcal{D}_{\lambda^*}^-$. Lemma \ref{limit1} yields that $u_n \to u \in S_{\lambda^*}$. It follows from Corollary \ref{submanifold}  that $H_\mu(u)<0$. However, this is a contradiction, since $H_\mu(u)=\lim H_\mu(u_n)=0$.

	\vskip.3cm
	{\bf Case 2: $\gamma<p$}
	\vskip.3cm
	 Arguing by contradiction as in the previous case, we find that $u_n \to u \in S_{\lambda^*}$ and by Corollary \ref{submanifold} it follows that $F(u)>\mu$. However, this is a contradiction, since $F(u)=\lim F(u_n)=\mu$. Thus the existence of $\varepsilon$ is guaranteed and the proof is complete.\qed\\
	 
\subsection{A mountain-pass critical point for $\lambda>\lambda^*$} All over this section we assume (H1),  (S), (C1), (C2), and $\mu_*< \lambda^*$. Recall that $$\Phi_\lambda(u_\lambda)=c_{\lambda,\mu}^+:=\inf_{\mathcal{N}_{\lambda,\mu}^+} \Phi_\lambda, $$ and \begin{equation*}
\mathcal{N}_{\lambda,\mu}^+=
\begin{cases}
\{u\in \mathcal{N}_\lambda^+: H_\mu(u)<0\} & \mbox{ if } \gamma>p,\\
\{u\in \mathcal{N}_\lambda^+: F(u)>\mu\} & \mbox{ if }\gamma<p.
\end{cases}
\end{equation*} for $\lambda\in(\lambda^*,\lambda^*+\varepsilon)$.  The next result follows from the proof of  Theorem \ref{t3}(1):
\begin{corollary}\label{MPG1}  For each $\lambda\in(\lambda^*,\lambda^*+\varepsilon)$ there holds $\displaystyle \inf_{\partial \mathcal{N}_{\lambda,\mu}^+}\Phi_\lambda>c_{\lambda,\mu}^+$.

\end{corollary}
Next we deal with the set
\begin{equation*}
	B_\delta:=\{tu:u\in \overline{\mathcal{N}_{\lambda,\mu}^+},\ t\in (1-\delta,1+\delta)\},
\end{equation*}
defined for $\delta>0.$

\begin{proposition}\label{MPG2} For each $\lambda\in(\lambda^*,\lambda^*+\varepsilon)$ there exists $\delta>0$ such that $\displaystyle \inf_{\partial B_\delta}\Phi_\lambda>c_{\lambda,\mu}^+$.
\end{proposition}
\begin{proof} Indeed, first we claim that there exists $c>0$ such that $\varphi_{\lambda,u}''(1)>c$ for all $u\in \overline{\mathcal{N}_{\lambda,\mu}^+}$. Since $\overline{\mathcal{N}_{\lambda,\mu}^+}\subset \mathcal{N}_\lambda^+$ it is clear that $\varphi_{\lambda,u}''(1)>0$ for $u\in \overline{\mathcal{N}_{\lambda,\mu}^+}$.
	\vskip.3cm
	{\bf Case 1: $\gamma>p$}
	\vskip.3cm
	Suppose, on the contrary, that there exists $(u_n) \subset \overline{\mathcal{N}_{\lambda,\mu}^+}$ such that  $\varphi_{\lambda,u_n}''(1)\to 0$. Since $\varphi_{\lambda,u_n}'(1)= 0$ for all $n$, it follows that $H_\lambda(u_n)\to 0$. We write $u_n=t_nw_n$ where $w_n\in\Theta_{\lambda}$ and
	\begin{equation}\label{eq333}
	t_n:=t_{\lambda}(w_n)=\left(\frac{H_{\lambda}(w_n)}{F(w_n)}\right)^{\frac{1}{\gamma-p}}.
	\end{equation}
	Thus $t_n^pH_\lambda(w_n)=H_\lambda(u_n) \to 0$ which, combined with \eqref{eq333}, implies that $H_\lambda(w_n)\to 0$. Therefore, we can assume that $w_n \rightharpoonup w \neq 0$ and moreover
	\begin{equation*}
	H_{\mu}(w_n)=H_\lambda(w_n)+(\lambda-\mu)P_2(w_n),
	\end{equation*}
	which implies that $H_\mu(w_n)>0$  for $n$ sufficiently large. This yields a contradiction, since $H_{\mu}(w_n)=t_n^{-p}H_{\mu}(u_n)\le 0$ for all $n$.

		\vskip.3cm
	{\bf Case 2: $\gamma<p$}
	\vskip.3cm
	It is enough to note that for $u \in  \overline{\mathcal{N}_{\lambda_n,\mu}^+}$ we have
	$\varphi_{\lambda,u}''(1)=(p-\gamma)F(u)\geq (p-\gamma)\mu$.\\
	
	  Thus  the claim is proved and  there exists $c>0$ such that $\varphi_{\lambda,u}''(1)>c$ for all $u\in \overline{\mathcal{N}_{\lambda,\mu}^+}$. This implies in particular that there exists $c>0$ such that $\|u\|\ge c$ for all $u\in \overline{\mathcal{N}_{\lambda,\mu}^+}$.  By making $c$ smaller if necessary, we can choose $\delta>0$ such that $\varphi_{\lambda,w}''(1)>c$ for all $w\in \overline{B_\delta}$. Therefore, for any $tu\in \overline{B_\delta}$ there holds
	\begin{eqnarray*}
		\Phi_\lambda(tu)-\Phi_\lambda(u)&=&\varphi_{\lambda,u}(t)-\varphi_{\lambda,u}(1) 
		=\varphi_{\lambda,u}'(1)(t-1)+\frac{1}{2}\varphi_{\lambda,u}''(\theta)(t-1)^2 \\
		&=&\frac{1}{2}\varphi_{\lambda,u}''(\theta)(t-1)^2 > \frac{c}{2}(t-1)^2,   
	\end{eqnarray*}
where $\theta\in(\min\{1,t\},\max\{1,t\})$. Now observe that
\begin{equation*}
	\partial B_\delta=\{tu: u\in  \partial \mathcal{N}_{\lambda,\mu}^+,\ t\in [-\delta,\delta]\}\cup \{su: u\in  \mathcal{N}_{\lambda,\mu}^+,\ s\in\{-\delta,\delta\}\}.
\end{equation*} 
From Corollary \ref{MPG1} we have, for $ t\in [-\delta,\delta]$ and  $ u\in  \partial \mathcal{N}_{\lambda,\mu}^+$, that
\begin{equation*}
	\Phi_\lambda(tu)>\Phi_\lambda(u)+\frac{c}{2}(t-1)^2>c_{\lambda,\mu}^+.
\end{equation*}
On the other hand, if $u\in  \mathcal{N}_{\lambda,\mu}^+$ and $s\in\{-\delta,\delta\}$, then
\begin{equation*}
	\Phi_\lambda(su)>\Phi_\lambda(u)+\frac{c}{2}(\pm \delta-1)^2\ge c_{\lambda,\mu}^++\frac{c}{2}(\pm \delta-1)^2,
\end{equation*}	
and the proof is complete.\\
\end{proof}

\noindent\textbf{Proof of Theorem \ref{t3} (2).}\\

Let $\delta>0$ be given by Proposition \ref{MPG2}. Since $B_\delta$ is bounded and $\Phi_\lambda$ is unbounded from below, we can find some $v_\lambda \in X \setminus B_\delta$ such that $\Phi_\lambda(v_\lambda)<c_\lambda$.

Therefore, setting 
\begin{equation*}
\Gamma_\lambda=\{\eta\in C([0,1],X): \eta(0)=u_\lambda,\ \eta(1)=v_\lambda\}
\end{equation*}
we infer that
\begin{equation*}
	d_\lambda=\inf_{\eta\in \Gamma_\lambda}\max_{t\in [0,1]}\Phi_\lambda(\eta(t))
\end{equation*}
 is a critical value of $\Phi_\lambda$ for $\lambda\in(\lambda^*,\lambda^*+\varepsilon)$.

\section{Applications}

Let us provide some applications of Theorems \ref{existenceouyang}, \ref{existencelambda^*} and \ref{t3}. Throughout this section $\Omega$ is a bounded domain of $\mathbb{R}^N$, with $N \ge 1$.

\subsection{Indefinite $p$-Laplace equations} \label{Secindefi}

We consider the functional associated to \eqref{eo}, i.e.
\begin{equation*}
\Phi_\lambda(u)=\frac{1}{p}\int_\Omega \left( |\nabla u|^p-\lambda (u^+)^p\right)-\frac{1}{\gamma}\int_\Omega f(x)(u^+)^{\gamma}, \quad u \in X= W_0^{1,p}(\Omega),
\end{equation*}
and some variations of it. Here $p>1$, $1<\gamma<p^*$ with $\gamma \neq p$, and $f \in L^{\infty}(\Omega)$. Recall that $p^*=\frac{Np}{N-p}$ if $p<N$ and $p^*=\infty$ if $p \ge N$.

Let $H_\lambda(u)=\int_\Omega \left( |\nabla u|^p-\lambda (u^+)^p\right)$, i.e. $P_1(u)=\int_\Omega  |\nabla u|^p =\|u\|^p$, $P_2(u)=\int_\Omega (u^+)^p$ ,and $F(u)=\int_\Omega f(x)(u^+)^{\gamma}$ for $u \in X$.
It is standard to check that $P_1$, $P_2$ and $F$ satisfy our basic assumptions. Note also that critical points of $\Phi_\lambda$ are nonnegative weak solutions of \eqref{eo}.
For this functional we have \begin{equation*}
\lambda^*:=\inf\left\{\frac{\int_\Omega |\nabla u|^p}{\int_\Omega (u^+)^p}: u \in W_0^{1,p}(\Omega) \setminus \{0\}, \int_\Omega f(x)(u^+)^{\gamma}=0 \right\}.
\end{equation*}
It is clear that $$\lambda^*\geq \inf\left\{\frac{\int_\Omega |\nabla u|^p}{\int_\Omega (u^+)^p}: u \in W_0^{1,p}(\Omega)  \right\}=\inf\left\{\frac{\int_\Omega |\nabla u|^p}{\int_\Omega |u|^p}: u \in W_0^{1,p}(\Omega)  \right\}=\lambda_1(p), $$  the first eigenvalue of $-\Delta_p$ in $W_0^{1,p}(\Omega)$. Indeed, the equality of the infima above follows from the fact that $\lambda_1(p)$ is achieved by  a positive eigenfunction $\varphi_1=\varphi_1(p)$, and the inequality
$\int_\Omega |\nabla u^+|^p \leq \int_\Omega |\nabla u|^p$ for any $u \in X$.

We shall assume that  $\int_\Omega f(x)\varphi_1^{\gamma}<0$, which clearly yields $\mu_*=\lambda_1(p)<\lambda^*$.
Let us show that (H1) also holds. Indeed, otherwise we would have $$\lambda^*>\Lambda:=\inf\left\{\frac{\int_\Omega |\nabla u|^p}{\int_\Omega (u^+)^p}: u \in W_0^{1,p}(\Omega) \setminus \{0\}, \int_\Omega f(x)(u^+)^{\gamma}\geq 0 \right\},$$
so $\Lambda$
would be achieved by some $u_0$ such that $\int_\Omega f(x)(u_0^+)^{\gamma}> 0$. In particular, $u_0$ minimizes $\frac{\int_\Omega |\nabla u|^p}{\int_\Omega (u^+)^p}$ over the open set $\{u \in W_0^{1,p}(\Omega): \int_\Omega f(x)(u^+)^{\gamma}> 0\}$. Thus $u_0$ would be a critical point of $\frac{\int_\Omega |\nabla u|^p}{\int_\Omega (u^+)^p}$, and consequently a nonnegative eigenfunction of $-\Delta_p$ associated to the eigenvalue $\Lambda$. However, the assumption $\int_\Omega f(x)\varphi_1^{\gamma}<0$ entails that $\Lambda>\lambda_1$. Finally, it is known that $\lambda_1$ is the only principal eigenvalue of $-\Delta_p$, i.e. it is the only eigenvalue associated to a nonnegative eigenfunction. So we reach a contradiction, and (H1) is proved.

The latter argument also shows that the condition $\int_\Omega f(x)\varphi_1^{\gamma}<0$  implies (C2), since $H_{\lambda^*}'(u)=0$ and $u^+ \neq 0$ means that $\lambda^*$ is a principal eigenvalue of $-\Delta_p$, i.e. $\lambda^*=\lambda_1(p)$, which is impossible. 

It is also clear that (S)  holds since $P_1(u)=\|u\|^p$ and $W_0^{1,p}(\Omega)$ is a uniformly convex space.

As for the (PS) condition, it suffices to show that any (PS) sequence is bounded (the (S+) property of the $p$-Laplacian gives then the desired conclusion). 
To this end, let us introduce the notation
\begin{equation*}
\Omega^0=\{x\in\Omega:f(x)=0\} \ \ \ \mbox{and} \ \ \ \Omega^+=\{x\in\Omega:f(x)>0\}.
\end{equation*}
More precisely, $\Omega_+$ is the largest open set where  $f>0$ a.e. We denote by $\operatorname{int}(\Omega^0)$ the interior of $\Omega^0$ and similarly we define  $\operatorname{int}(\Omega^0\cup \Omega^+)$. Given  an open, bounded and smooth set $U$, we denote by $(\lambda_1(p,U),\phi_1(U))$ the first eigenpair of $(-\Delta_p,U)$. We shall assume that $\operatorname{int}(\Omega^0)$ is smooth, so that the following property holds:
\begin{itemize}
\item [($h_0$)]If $v \in W_0^{1,p}(\Omega)$ and $fv \equiv 0$ then $v \in W_0^{1,p}(\operatorname{int}(\Omega^0))$.
\end{itemize}
This property holds, for instance, if $\operatorname{int}(\Omega^0)$ is a $p$-stable set, in the capacity sense (see e.g. \cite[Proposition 11]{CRQ}).

For $\gamma>p$ we shall prove that (PS) holds for	$\lambda<\lambda_1(p,\operatorname{int}(\Omega^0))$.
Indeed, let $(u_n) \subset X$ be such that $(\Phi_\lambda(u_n))$ 
is bounded and $\Phi_\lambda'(u_n)\to 0$. Assume by contradiction that $\|u_n\|\to \infty$ and $v_n:=\frac{u_n}{\|u_n\|} \rightharpoonup v$ in $X$. 
Since  $\Phi_\lambda'(u_n)\phi \to 0$ we find that
$\int_\Omega f(v^+)^{\gamma-1} \phi=\lim \int_\Omega f(v_n^+)^{\gamma-1} \phi=0$  for every $\phi \in X$. It follows that $fv^+ \equiv 0$ so $v^+\in W_0^{1,p}(\operatorname{int}(\Omega^0))$. Moreover, combining the fact that $(\Phi_\lambda(u_n))$ 
is bounded and $|\Phi_\lambda'(u_n)u_n|\le \varepsilon_n \|u_n\|$ with $\varepsilon_n \to 0$, we derive that $H_\lambda(v_n) \to 0$, so $v^+ \neq 0$ and $H_\lambda(v) \le 0$. It follows that $\lambda \ge  \lambda_1(p,\operatorname{int}(\Omega^0))$, a contradiction.

Now, if $\gamma<p$ then  we have now $H_\lambda'(v)\phi=\lim H_\lambda'(v_n)\phi=0$ for every $\phi \in X$. From the boundedness of $(\Phi_\lambda^+(u_n))$ it follows that $H_\lambda(v_n)\to 0$, which yields $v^+ \not \equiv 0$. Thus
$-\Delta_p v=\lambda (v^+)^{p-1}$ so that $v \ge 0$, and therefore $\lambda=\lambda_1(p)$ and $v$ is an eigenfunction associated to $\lambda_1(p)$.
Now, the fact that $(\Phi_\lambda(u_n))$ 
is bounded and $|\Phi_\lambda'(u_n)u_n|\le \varepsilon_n \|u_n\|$ with $\varepsilon_n \to 0$, yields that $\int_\Omega f(v^+)^{\gamma}=\lim \int_\Omega f(v_n^+)^{\gamma}=0$, which contradicts the assumption $\int_\Omega f(x)\varphi_1^{\gamma}<0$. Therefore in this case (PS) holds for $\lambda \neq \lambda_1(p)$.

Finally we shall prove that (C1) holds under the following additional condition:
\begin{itemize}
	\item[($f_0$)] If  $|\Omega^0|>0$ then $\operatorname{int}(\Omega^0)$ is an open, bounded and smooth set, satisfying
	\begin{equation*}
	\lambda^*<\lambda_1(p,\operatorname{int}(\Omega^0)).
	\end{equation*}
\end{itemize}

\begin{proposition}\label{f01} Assume $(h_0)$ and $(f_0)$. Then  (C1) holds true.
\end{proposition}
\begin{proof} Indeed, suppose on the contrary that $F'(u)=0$ and $u$ achieves $\lambda^*$. It follows that $f(u^+)^{\gamma-1}\equiv 0$ and hence  $u^+\in W_0^{1,p}(\operatorname{int}(\Omega^0))$ by $(h_0)$. It follows that
	\begin{equation}\label{eqo1}
		\lambda_1(p,\operatorname{int}(\Omega^0))\le \frac{\int_{\operatorname{int}(\Omega^0)} |\nabla u^+|^p}{\int_{\operatorname{int}(\Omega^0)} (u^+)^p}\le  \frac{\int_{\operatorname{int}(\Omega^0)} |\nabla u|^p}{\int_{\operatorname{int}(\Omega^0)} (u^+)^p}=\lambda^*,
	\end{equation}
which contradicts $(f_0)$. Therefore $F'(u)\neq 0$.

	
\end{proof}

Let us discuss on the condition $(f_0)$. It is clear that 
	\begin{equation*}
	\lambda^*\le \lambda_1(p,\operatorname{int}(\Omega^0)).
	\end{equation*}
	We shall provide some conditions that ensure the inequality.

	\begin{corollary}\label{f02} Suppose that $\operatorname{int}(\Omega^0)$ and $\operatorname{int}(\overline{\Omega^0}\cup \overline{\Omega^+})$ are  bounded and smooth sets. If $(h_0)$ holds, and
		\begin{equation*}
		\lambda_1(p,\operatorname{int}(\overline{\Omega^0}\cup \overline{\Omega^+}))<\lambda_1(p,\operatorname{int}(\Omega^0)),
		\end{equation*}
	then (C1) holds true.
	
\end{corollary}
\begin{proof} Indeed, it is clear that
	\begin{equation*}
		\lambda^*\le 	\lambda_1(p,\operatorname{int}(\overline{\Omega^0}\cup \overline{\Omega^+})),
	\end{equation*}
and hence $(f_0)$ is satisfied, which implies, by Proposition \ref{f01}, the condition (C1).
	
\end{proof}
\begin{corollary}  Suppose that $\operatorname{int}(\Omega^0)$ and $\operatorname{int}(\overline{\Omega^0}\cup \overline{\Omega^+})$ are bounded smooth domains. If $(h_0)$ holds and $\operatorname{int}(\Omega^0)$ is a proper subset of $\operatorname{int}(\overline{\Omega^0}\cup \overline{\Omega^+})$, then (C1) holds true. In particular, if there exists a ball $B\subset \operatorname{int}(\overline{\Omega^0}\cup \overline{\Omega^+})$ such that $B\cap \Omega^0\neq \emptyset$ and $B\cap \Omega^+\neq \emptyset$, then (C1) holds true.
\end{corollary} 
\begin{proof}Indeed, if $\operatorname{int}(\Omega^0)$ is a proper subset of $\operatorname{int}(\overline{\Omega^0}\cup \overline{\Omega^+})$, then (see e.g. \cite[Proposition 4.4]{C})
	\begin{equation*}
	\lambda_1(p,\operatorname{int}(\overline{\Omega^0}\cup \overline{\Omega^+}))<\lambda_1(p,\operatorname{int}(\Omega^0)),
\end{equation*}	
and from Corollary \ref{f02} we obtain (C1). To conclude, it is clear that if such a ball exists, then $\operatorname{int}(\Omega^0)$ is a proper subset of $\operatorname{int}(\overline{\Omega^0}\cup \overline{\Omega^+})$.\\
\end{proof}

\begin{remark}
	It is worth pointing out that if $\gamma>p$, $f \leq 0$ and $|\Omega_0|>0$ then  one may still consider $c_{\lambda^*}^+$. However, condition (C1) fails in this case, since  $F(u)=0$ clearly implies $F'(u)=0$. This fact is not a technical issue, since one can show that $\Phi_\lambda$ has no nontrivial critical point for $\lambda\ge \lambda_1(p,\operatorname{int}(\Omega^0))$.
	\end{remark}

Thus we infer the following result:
\begin{corollary}
Let $p>1$ and $\gamma \in (1,p^*)$ with $\gamma \ne p$. Assume that $f \in L^{\infty}(\Omega)$ satisfies $(h_0),(f_0)$, and $\int_\Omega f(x)\varphi_1^{\gamma}<0$. Then the conclusions of Theorems \ref{existenceouyang}, \ref{existencelambda^*} and \ref{t3} hold true.
\end{corollary}

The previous result has been established for $\gamma>p$ in \cite{IS}, whereas for $\gamma<p$ it extends the analysis carried  out in \cite{B} for $p=2$ and in \cite{KRQU} for $p>1$, both dealing with $\lambda<\lambda^*$.

Let us consider now the Neumann problem
\begin{equation}
\label{np} -\Delta_p u =\lambda |u|^{p-2}u +f(x)|u|^{\gamma-2}u, \quad u \in W^{1,p}(\Omega),
\end{equation}
In comparison with the functional of the Dirichlet problem, a slight modification is needed. Namely, we set
\begin{equation}
\Phi_\lambda(u)=\frac{1}{p}\int_\Omega \left( |\nabla u|^p+(u^-)^p-\lambda (u^+)^p\right)-\frac{1}{\gamma}\int_\Omega f(x)(u^+)^{\gamma},
\end{equation}
for  $u \in W^{1,p}(\Omega)$. It is clear that critical poins of this functional are nonnegative and thus solutions of \eqref{np}.
We have then
\begin{equation*}
\lambda^*:=\inf\left\{\frac{\int_\Omega |\nabla u|^p+(u^-)^p}{\int_\Omega (u^+)^p}: u \in W^{1,p}(\Omega) \setminus \{0\}, \int_\Omega f(x)(u^+)^{\gamma}=0 \right\}.
\end{equation*}
As in the previous problem, one can show that $$\lambda^*\geq \inf\left\{\frac{\int_\Omega |\nabla u|^p+(u^-)^p}{\int_\Omega (u^+)^p}: u \in W^{1,p}(\Omega)  \right\}=\inf\left\{\frac{\int_\Omega |\nabla u|^p}{\int_\Omega |u|^p}: u \in W^{1,p}(\Omega)  \right\}=0, $$
and the inequality holds if $\int_\Omega f<0$. This condition also yields $\mu_*=0<\lambda^*$, (H1), and (C2). It is also clear that (S) is satisfied and proceeding as in the Dirichlet case one can show that (PS) holds for $\lambda \neq 0$.
Finally, we note that $(f_0)$ is weaker than in the Dirichlet case, since the infimum in the definition of $\lambda^*$ is taken over $W^{1,p}(\Omega)$ whereas $\lambda_1(p,\operatorname{int}(\Omega^0))$ remains unchanged. 

 A similar analysis applies to the functionals

\begin{equation}\label{np1}
\Phi_\lambda^1(u)=\frac{1}{p}\int_\Omega \left( |\nabla u|^p+(u^-)^p-\lambda (u^+)^p\right)-\frac{1}{\gamma}\int_{\partial \Omega} f(x)(u^+)^{\gamma},
\end{equation}
and
\begin{equation}\label{np2}
\Phi_\lambda^2(u)=\frac{1}{p}\left( \int_\Omega |\nabla u|^p+(u^-)^p-\lambda \int_{\partial \Omega} (u^+)^p\right)-\frac{1}{\gamma}\int_\Omega f(x)(u^+)^{\gamma},
\end{equation}

defined in $X=W^{1,p}(\Omega)$. These functionals are respectively associated to the problems
\[
\left\{
\begin{array}
[c]{lll}%
-\Delta_p u =\lambda |u|^{p-2}u & \mathrm{in} & \Omega,\\
|\nabla u|^{p-2}\partial_n u=f(x)|u|^{\gamma-2}u  & \mathrm{on} & \partial\Omega.
\end{array}
\right.   
\quad \mbox{and} \quad 
\left\{
\begin{array}
[c]{lll}%
-\Delta_p u =f(x)|u|^{\gamma-2}u & \mathrm{in} & \Omega,\\
|\nabla u|^{p-2}\partial_n u=\lambda |u|^{p-2}u  & \mathrm{on} & \partial\Omega.
\end{array}
\right.   
\]
We refer to \cite{A,CT,KRQU1,KRQU2} for previous results on \eqref{np}, \eqref{np1}, and \eqref{np2} with $1<\gamma<p=2$. 
\medskip
\subsection{$(p,q)$-Laplacian problems} \label{pqlapl}
We consider now the functional 
\begin{equation}\label{fpq}
 \Phi_\lambda(u)=\frac{1}{p}\int_\Omega \left(|\nabla u|^p- \lambda(u^+)^p \right)+\frac{1}{q}\int_\Omega \left(|\nabla u|^q- \beta(x)(u^+)^q \right) \quad u \in W_0^{1,p}(\Omega),
\end{equation}
whose critical points are nonnegative solutions of the $(p,q)$-Laplacian problem
 \begin{equation} \label{pq}
-\Delta_p u -\Delta_q u = \lambda |u|^{p-2}u+\beta(x) |u|^{q-2}u, \quad u \in W_0^{1,p}(\Omega),
\end{equation}
where $1<q<p$, $\lambda \in \mathbb{R}$ and $\beta \in L^{\infty}(\Omega)$ is nonnegative and nontrivial. This problem, with $\beta$ constant, has been recently studied in \cite{BT,BT2,BT3}.

Here $H_\lambda(u)=\int_\Omega \left( |\nabla u|^p-\lambda (u^+)^p\right)$, whereas $F$ is given now by 
$$F(u)=-\int_\Omega \left(|\nabla u|^q- \beta(x)(u^+)^q \right),$$
so that
\begin{equation*}
\lambda^*=\lambda^*(\beta):=\inf\left\{\frac{\int_\Omega |\nabla u|^p}{\int_\Omega (u^+)^p}: u \in W_0^{1,p}(\Omega) \setminus \{0\}, \int_\Omega \left(|\nabla u|^q- \beta(x)(u^+)^q \right) =0 \right\}.
\end{equation*}
We still have $\lambda^* \ge \lambda_1(p)$, and the inequality holds if, and only if, $F(\varphi_p)<0$,  i.e.
$\int_\Omega \left(|\nabla \varphi_p|^q- \beta(x)\varphi_p^q \right)>0$, where $\varphi_p:=\varphi_1(p)$. 
Arguing as in the previous subsection, we can show that this condition implies (H1) and (C2).

 To check (C1), assume by contradiction that $\lambda^*$ is achieved by some $u$ such that $F'(u)=0$. It follows that
$\int_\Omega |\nabla u|^{q-2} \nabla u \nabla \phi=\int_\Omega \beta(x)(u^+)^{q-1}\phi$ for every $\phi \in X$. which yields $u \ge 0$. Moreover, even though this equation holds in $(W_0^{1,p}(\Omega))^*$, one can show that $\lambda_1(\beta,q)=1$, where
$$\lambda_1(\beta,q):=\inf\left\{\frac{\int_\Omega |\nabla u|^q}{\int_\Omega \beta(x)|u|^q}: u \in W_0^{1,q}(\Omega)  \right\}.$$ Thus (C1) holds if $\lambda_1(\beta,q) \ne 1$. Note also that $F$ takes positive values if, and only if, $\lambda_1(\beta,q)<1$. Finally, proceeding as in the previous subsection (in the case $\gamma<p$), one can show that (PS) holds for any $\lambda\ne \lambda_1(p)$ if $F(\varphi_p)<0$.
Summing up, we derive the following result: 

\begin{corollary}
	Let $1<q<p$, $\lambda \in \mathbb{R}$ and $\beta \in L^{\infty}(\Omega)$ be nonnegative with $\lambda_1(\beta,q)<1$.
	\begin{enumerate}
\item If $\lambda <\lambda_1(p)$ then $c_\lambda^+$ is achieved, i.e. \eqref{pq} has a nonnegative solution $u_+ \in \mathcal{N}_\lambda^+$.
\item If $\int_\Omega \left(|\nabla \varphi_p|^q- \beta(x)\varphi_p^q \right)>0$ then $\lambda^*>\lambda_1(p)$, and \eqref{pq} has:
\begin{enumerate}
\item one nonnnegative solution $u_+ \in \mathcal{N}_\lambda^+$ for $\lambda=\lambda_1(p)$.
\item two nonnegative solutions $u_+ \in \mathcal{N}_\lambda^+$ and $u_- \in \mathcal{N}_\lambda^-$ for $\lambda_1(p)<\lambda<\lambda^*$. Moreover, there exists  $\varepsilon>0$ such that \eqref{pq} has two nonnegative solutions for $\lambda^* \leq \lambda<\lambda^* +\varepsilon$.
\end{enumerate}
	\end{enumerate}		
\end{corollary}

The previous result extends \cite[Theorem 2.7]{BT2}, which deals with $\beta$ constant. In this case, the condition $\lambda_1(\beta,q) < 1$ reads as $\beta>\lambda_1(q)$, whereas $\int_\Omega \left(|\nabla \varphi_p|^q- \beta(x)\varphi_p^q \right)>0$ becomes now $\beta<\beta_*:=\frac{\int_\Omega |\nabla \varphi_p|^q}{\int_\Omega \varphi_p^q}$. Let us note that in \cite[Theorem 2.7]{BT2} the roles of $\lambda$ and $\beta$ are interchanged (see \cite[Section 3.2]{FRQS} for more details).

Similarly to the problem of the previous section, we may also consider the functional
\begin{equation}
\Phi_\lambda(u)=\frac{1}{p}\int_\Omega \left(|\nabla u|^p+(u^-)^p-\lambda(u^+)^p \right)+\frac{1}{q}\int_\Omega \left(|\nabla u|^q- \beta(x)(u^+)^q \right) \quad u \in W_0^{1,p}(\Omega),
\end{equation}
on $X=W^{1,p}(\Omega)$. In this case, we have $\lambda_1(p)=0$ and $\varphi_p$ is a positive constant. Hence the condition $\int_\Omega \left(|\nabla \varphi_p|^q- \beta(x)\varphi_p^q \right)>0$ reads $\int_\Omega \beta<0$, so that we need $\beta^- \not \equiv 0$. On the other hand, $F$ take positive values only if  $\beta^+ \not \equiv 0$. Thus $\beta$ has to change sign. We derive then the following result on the problem
\begin{equation} \label{pqn}
-\Delta_p u -\Delta_q u = \lambda |u|^{p-2}u+\beta(x) |u|^{q-2}u, \quad u \in W^{1,p}(\Omega).
\end{equation}
\medskip
\begin{corollary}
	Let $1<q<p$, $\lambda \in \mathbb{R}$ and $\beta \in L^{\infty}(\Omega)$ be  sign-changing annd such that 
$$	\inf\left\{\frac{\int_\Omega |\nabla u|^q}{\int_\Omega \beta(x)|u|^q}: u \in W^{1,q}(\Omega), \int_\Omega \beta(x)|u|^q>0 \right\}<1.$$
	\begin{enumerate}
		\item If $\lambda <0$ then $c_\lambda^+$ is achieved, i.e. \eqref{pqn} has a nonnegative solution $u_+ \in \mathcal{N}_\lambda^+$.
		\item If $\int_\Omega \beta<0$ then $\lambda^*>0$, and \eqref{pqn} has:
		\begin{enumerate}
			\item one nonnnegative solution $u_+ \in \mathcal{N}_\lambda^+$ for $\lambda=0$.
			\item two nonnegative solutions $u_+ \in \mathcal{N}_\lambda^+$ and $u_- \in \mathcal{N}_\lambda^-$ for $0<\lambda<\lambda^*$. Moreover, there exists  $\varepsilon>0$ such that \eqref{pqn} has two nonnegative solutions for $\lambda^* \leq \lambda<\lambda^* +\varepsilon$.
		\end{enumerate}
	\end{enumerate}		
\end{corollary}
\begin{remark}\label{unboundedexamples} \rm As mentioned in Remark \ref{unbogamma<p}, we have $c_\lambda^+=-\infty$ if $\gamma<p$ and $c_\lambda^-=0$ if $\gamma>p$ and $\lambda>\lambda^*$, for the functional of Section 3.1,  and $c_\lambda^+=-\infty$ if $\lambda>\lambda^*$, for the functional of Section 3.2. Indeed, in the first case one may argue as in the proof of \cite[Lemma 4]{B} to show that $c_\lambda^+=-\infty$ for $\lambda>\lambda^*$ if $\gamma<p$ . To this end, it suffices to prove that if $\lambda>\lambda^*$, then there exists $u$ such that $H_\lambda(w)<0$ and $F(w)>0$. We choose $u\neq 0$ such that $H_{\lambda^*}(u)=F(u)=0$. It is clear that $H_\lambda(u)<0$. Choose $r>0$ such that $H_\lambda(w)<0$ for all $w\in B$, where $B$ is the ball centered at $u$, with radius $r$. If $F(w)\le 0$ for all $w\in B$, then $u$ is a local maximizer of $F$, so that $F'(u)=0$, which contradicts (C1). Therefore there exists $w$ such that $H_\lambda(w)<0$ and $F(w)>0$ which completes the proof.  
	For the functional of Section 3.2 we refer to \cite[Theorem 2.5(ii)]{BT2}.
	
\end{remark}
\medskip
\subsection{Kirchhoff type problems}
We consider now the functional 
\begin{equation}\label{fkir}
	\Phi_\lambda(u)=\frac{1}{2}\int_\Omega \left(a|\nabla u|^2- \lambda(u^+)^2 \right)+\frac{b}{4}\left(\int_\Omega |\nabla u|^2\right)^2- \frac{1}{4}\int_\Omega\beta(x)(u^+)^4  \quad u \in W_0^{1,2}(\Omega),
\end{equation}
whose critical points are nonnegative solutions of the Kirchhoff type  problem
\begin{equation} \label{kir}
	-\left(a+b\int_\Omega |\nabla u|^2\right) \Delta u = \lambda |u|^{p-2}u+\beta(x) |u|^{q-2}u, \quad u \in H_0^1(\Omega).
\end{equation}
Here $a,b>0$, $\lambda \in \mathbb{R}$ and $\beta \in L^{\infty}(\Omega)$ is nonnegative and nontrivial, and $N \leq 3$. This problem, with $\beta$ constant, has been recently investigated in \cite{CO,SS}.

Here $H_\lambda(u)=\int_\Omega \left( a|\nabla u|^2-\lambda (u^+)^2\right)$, whereas 
$F(u)=-b\left(\int_\Omega |\nabla u|^2\right)^2+\int_\Omega\beta(x)(u^+)^4 $,
so that
\begin{equation*}
	\lambda^*=\lambda^*(\beta):=\inf\left\{\frac{\int_\Omega |\nabla u|^2}{\int_\Omega (u^+)^2}: u \in H_0^1(\Omega) \setminus \{0\}, b\left(\int_\Omega |\nabla u|^2\right)^2-\int_\Omega\beta(x)(u^+)^4 =0 \right\}.
\end{equation*}
We still have $\lambda^* \ge \lambda_1:=\lambda_1(2)$, and the inequality holds if, and only if, $F(\varphi_1)<0$,  i.e.
$b\left(\int_\Omega |\nabla\varphi_1|^2\right)^2-\int_\Omega\beta(x)(\varphi_1^+)^4 >0$, where $\varphi_1:=\varphi_1(2)$. 
Arguing as in the previous subsection, we can show that this condition implies (H1) and (C2).

To check (C1), assume by contradiction that $\lambda^*$ is achieved by some $u$ such that $F'(u)=0$. It follows that
$b\int_\Omega |\nabla u|^2\int_\Omega \nabla u \nabla \phi=\int_\Omega \beta(x)(u^+)^3\phi$ for every $\phi \in X$, which yields $u \ge 0$. Therefore $\lambda_1(\beta)=1$, where
$$\lambda_1(\beta):=\inf\left\{\frac{b\left(\int_\Omega |\nabla u|^2\right)^2}{\int_\Omega \beta(x)|u|^4}: u \in H_0^1(\Omega)  \right\}.$$ Thus (C1) holds if $\lambda_1(\beta) \ne 1$. Note also that $F$ takes positive values if, and only if, $\lambda_1(\beta)<1$. Finally, proceeding as in the previous subsection (in the case $\gamma<p$), one can show that (PS) holds for any $\lambda\ne \lambda_1$ if $F(\varphi_1)<0$.
Summing up, we derive the following result: 

\begin{corollary}
	Let $\lambda \in \mathbb{R}$ and $\beta \in L^{\infty}(\Omega)$ be nonnegative with $\lambda_1(\beta)<1$.
	\begin{enumerate}
		\item If $\lambda <\lambda_1$ then $c_\lambda^+$ is achieved, i.e. \eqref{kir} has a nonnegative solution $u_+ \in \mathcal{N}_\lambda^+$.
		\item If $b\left(\int_\Omega |\nabla\varphi_1|^2\right)^2-\int_\Omega\beta(x)(\varphi_1^+)^4 >0$ then $\lambda^*>\lambda_1$, and \eqref{kir} has:
		\begin{enumerate}
			\item one nonnnegative solution $u_+ \in \mathcal{N}_\lambda^+$ for $\lambda=\lambda_1$.
			\item two nonnegative solutions $u_+ \in \mathcal{N}_\lambda^+$ and $u_- \in \mathcal{N}_\lambda^-$ for $\lambda_1<\lambda<\lambda^*$. Moreover, there exists  $\varepsilon>0$ such that \eqref{kir} has two nonnegative solutions for $\lambda^* \leq \lambda<\lambda^* +\varepsilon$.
		\end{enumerate}
	\end{enumerate}		
\end{corollary}

The previous result extends \cite[Theorem 2]{SS}, which deals with $\beta$ constant. In this case, the condition $\lambda_1(\beta) < 1$ reads as $\beta>b\mu_1$, where $\mu_1$ is defined by
\begin{equation*}
	\mu_1=\inf\left\{\frac{b\left(\int_\Omega |\nabla u|^2\right)^2}{\int_\Omega|u|^4}: u \in H_0^1(\Omega) \right\}.
\end{equation*}
Let us observe that $\mu_1$ is the first eigenvalue
of the problem 
\begin{equation} \label{kir}
-b\left(\int_\Omega |\nabla u|^2\right) \Delta u = \lambda |u|^2 u, \quad u \in H_0^1(\Omega),
\end{equation}
see \cite{D}. Moreover $b\left(\int_\Omega |\nabla\varphi_1|^2\right)^2-\int_\Omega\beta(x)(\varphi_1^+)^4 >0$ becomes now $\beta<\beta_*:=\frac{b\left(\int_\Omega |\nabla \varphi_1|^2\right)^2}{\int_\Omega \varphi_1^4}$.

\medskip

\subsection*{Acknowledgements}
The authors are thankful to V. Bobkov for pointing out an error in a former version of Proposition \ref{coercivi>lambda^*}(1).

\subsection*{Data availability}
Data sharing not applicable to this article as no datasets were generated or analysed during the current study.


\begin{thebibliography}{99}   
	 
\bibitem {A} S. Alama, \textit{Semilinear elliptic equations with sublinear
	indefinite nonlinearities}, Adv. Differential Equations \textbf{4} (1999), 813--842.
	 
\bibitem {AT}S. Alama, G. Tarantello, On semilinear elliptic equations	with indefinite nonlinearities, Calc. Var. Partial Differential Equations	\textbf{1} (1993), 439--475.	

\bibitem{AS} J. C. de Albuquerque, K. Silva, On the extreme value of the Nehari manifold method for a class of Schrödinger equations with indefinite weight functions. J. Differential Equations 269 (2020), no. 7, 5680–5700.

\bibitem{Am} A. Ambrosetti, Critical points and nonlinear variational problems. Mémoires de la S. M. F. 2e série, tome 49 (1992).

\bibitem{BDH} P. A. Binding, P. Drábek, Y. X. Huang, On Neumann boundary value problems for some quasilinear elliptic equations. Electron. J. Differential Equations (1997), No. 05, approx. 11 pp.

\bibitem {BD}I. Birindelli, F. Demengel, Existence of solutions for
semi-linear equations involving the p-Laplacian: the non coercive case. Calc.
Var. Partial Differential Equations \textbf{20} (2004), 343--366.

\bibitem{BT} V. Bobkov, M. Tanaka,   On positive solutions for $(p,q)$-Laplace equations with two parameters. Calculus of Variations and Partial Differential Equations, 54(3), 3277-3301.

\bibitem{BT2} V. Bobkov, M. Tanaka,  Remarks on minimizers for (p,q)-Laplace equations with two parameters.
Communications on Pure and Applied Analysis, 17(3), (2018) 1219-1253.

\bibitem{BT3} V. Bobkov, M. Tanaka,  Multiplicity of positive solutions for $(p,q)$-Laplace equations with two parameters. (2020), to appear in Comm. Contemp. Math.

\bibitem{B} K.J. Brown, The Nehari manifold for a semilinear elliptic
	equation involving a sublinear term. Calc. Var. Partial Differential
	Equations \textbf{22} (2005), 483--494.
	

\bibitem{BZ}	 K.J. Brown, Y. Zhang,  The Nehari manifold for a semilinear elliptic equation with a sign-changing weight function. J. Differential Equations 193 (2003), no. 2, 481–499. 

\bibitem {CT}J. Chabrowski, C.\ Tintarev, An elliptic problem with
	an indefinite nonlinearity and a parameter in the boundary condition, NoDEA
Nonlinear Differ. Equ. Appl.\ \textbf{21} (2014), 519--540.

\bibitem{CO}  B. Chen, Z. Q. Ou, Existence and bifurcation behavior of positive solutions for a class of Kirchhoff-type problems. Comput. Math. Appl. 77(10), 2859–2866 (2019). 

\bibitem{C}  M. Cuesta, Eigenvalue problems for the p-Laplacian with indefinite weights. Electron. J. Differential Equations 2001, No. 33, 9 pp.

\bibitem{CRQ} M. Cuesta, H. Ramos Quoirin,  A weighted eigenvalue problem for the p-Laplacian plus a potential. NoDEA Nonlinear Differential Equations Appl. 16 (2009), no. 4, 469–491.

\bibitem{D} G. Dai,  Eigenvalues, global bifurcation and positive solutions for a class of nonlocal elliptic equations.Topol. Methods Nonlinear Anal. 48(1), 213–233 (2016)

\bibitem{DP} P. Drábek, S. I. Pohozaev, 
Positive solutions for the p-Laplacian: application of the fibering method. 
Proc. Roy. Soc. Edinburgh Sect. A 127 (1997), no. 4, 703–726. 
	
\bibitem{FRQS} G. M. Figueiredo, H. Ramos Quoirin, K. Silva, Ground states of elliptic problems over cones. Calc. Var. Partial Differential Equations 60 (2021), no. 5, Paper No. 189, 29 pp.

\bibitem {I}Y. Il'yasov, On positive solutions of indefinite elliptic
equations, Comptes Rendus de l'Acad\'{e}mie des Sciences-Series I-Mathematics
\textbf{333} (2001), 533--538.
	

\bibitem{IS} Y. Il'yasov, K. Silva, On branches of positive solutions for p-Laplacian problems at the extreme value of the Nehari manifold method, Proc. Amer. Math. Soc. 146 (2018), no. 7, 2925--2935. 


\bibitem {KRQU1}U.\ Kaufmann, H.\ Ramos Quoirin, K.\ Umezu,
\textit{Nonnegative solutions of an indefinite sublinear Robin problem I:
	positivity, exact multiplicity, and existence of a subcontinuum}, Annali di
Matematica 199, 2015–2038 (2020).


\bibitem {KRQU2}U.\ Kaufmann, H.\ Ramos Quoirin, K.\ Umezu,
\textit{Nonnegative solutions of an indefinite sublinear Robin problem II:
	local and global exactness results,} to appear in Israel J. Math. 

\bibitem {KRQU}U.\ Kaufmann, H.\ Ramos Quoirin, K.\ Umezu,
Uniqueness and positivity issues in a quasilinear indefinite problem. Calc. Var. Partial Differential Equations 60 (2021), no. 5, Paper No. 187, 21 pp.	

\bibitem{Ou} T. Ouyang,  On the positive solutions of semilinear equations 
	$\Delta u+\lambda u+hu^p=0$ on compact manifolds. II.  Indiana Univ. Math. J.  40  
(1991),  no. 3, 1083--1141. 

\bibitem{SS} K. Silva, S. M. Sousa,  Finer analysis of the Nehari set associated to a class of Kirchhoff-type equations. SN Partial Differ. Equ. Appl. 1, 43 (2020).

\bibitem{SW} A. Szulkin, T. Weth, The method of Nehari manifold. Handbook of nonconvex analysis and applications, 597–632, Int. Press, Somerville, MA, (2010).

\end{thebibliography}
\end{document}